\documentclass[a4paper]{amsart}
\usepackage[usenames,dvipsnames]{xcolor}
\usepackage{amsmath,amsthm,amssymb,latexsym,epic,bbm,mathbbol,comment, mathrsfs}
\usepackage{tikz-cd}
\usetikzlibrary{matrix}
\usepackage{graphicx,enumitem,stmaryrd,color}
\usepackage[all,2cell]{xy}
\xyoption{2cell}
\UseAllTwocells
\usepackage[active]{srcltx}
\usepackage[parfill]{parskip}
\usepackage{url}
\usepackage{hyperref}
\usepackage{mathtools}
\usepackage{scalerel}

\usepackage{tikz}
\usetikzlibrary{fit}
\usetikzlibrary{calc}
\tikzset{%
  highlight/.style={rectangle,rounded corners,fill=red!55,draw,
    fill opacity=0.3,thick,inner sep=0pt}
}

\newcommand{\tikzmarkup}[1]{\tikz[overlay,remember picture] \node [above = 2] (#1) {};}
\newcommand{\tikzmarkdown}[1]{\tikz[overlay,remember picture] \node (#1) {};}
\newcommand{\DrawRedBox}[1][]{%
    \tikz[overlay,remember picture]{
    \node[highlight,fit=(left.north west) (right.south east)] (#1) {};
      }
}
\newcommand{\DrawRedBoxZ}[1][]{%
\tikzset{%
  highlight/.style={rectangle,rounded corners,fill=red!55,draw,
    fill opacity=0.3,thick,inner sep=0pt}
}
    \tikz[overlay,remember picture]{
    \node[highlight,fit=(left2.north west) (right2.south east)] (#1) {};
      }
}

\definecolor{blue(pigment)}{rgb}{0.2, 0.2, 0.6}

\newtheorem{theorem}{Theorem}[section]
\newtheorem*{theorem*}{Theorem}
\newtheorem{lemma}[theorem]{Lemma}
\newtheorem{proposition}[theorem]{Proposition}
\newtheorem{corollary}[theorem]{Corollary}

\theoremstyle{definition}
\newtheorem{definition}[theorem]{Definition}
\newtheorem*{definition*}{Definition}
\newtheorem{remark}[theorem]{Remark}

\newtheorem{example}[theorem]{Example}

\font\sc=rsfs10
\newcommand{\csym}[1]{\sc\mbox{#1}\hspace{1.0pt}}

\font\scc=rsfs7
\newcommand{\ccf}[1]{\scc\mbox{#1}\hspace{0.5pt}}

\newcommand{\cccsym}[1]

\newcommand{\on}[1]{\operatorname{#1}}

\newcommand{\setj}[1]{\left\{ #1 \right\}}
\newcommand{\rr}[1]{\left\llbracket #1 \right\rrbracket}
\newcommand{\xiso}{\xrightarrow{\sim}}
\newcommand{\idem}{complete set of pairwise orthogonal, primitive idempotents}
\newcommand{\zigzag}{\mathrm{Z}_{\rightleftarrows}}
\newcommand{\starv}[1]{\zigzag(S_{#1})}

\DeclareMathAlphabet\EuRoman{U}{eur}{m}{n}
\SetMathAlphabet\EuRoman{bold}{U}{eur}{b}{n}

\begin{document}
\pagestyle{plain}
\title{Simple transitive $2$-representations of $2$-categories associated to self-injective cores}
\author{Mateusz Stroi{\' n}ski}

\begin{abstract}
  Given a finite dimensional algebra $A$, we consider certain sets of idempotents of $A$, called self-injective cores, to which we associate $2$-sub\-ca\-te\-go\-ries of the $2$-category \scaleobj{0.82}{\csym{C}_{A}} of projective bimodules over $A$. We classify the simple transitive $2$-representations of such $2$-subcategories, vastly generalizing the main result of \cite{Zi}.
\end{abstract}
\maketitle
\tableofcontents

\section{Introduction}
 Motivated by the extensive use of categorical actions in various areas of mathematics, notably in famous results such as the introduction of Khovanov homology in \cite{Kh} and the proof of Brou{\' e}'s abelian defect group conjecture for symmetric groups in \cite{CR}, Mazorchuk and Miemietz initiated a systematic study of $2$-representations of the so-called finitary $2$-categories in the series of papers \cite{MM1}, \cite{MM2}, \cite{MM3}, \cite{MM4}, \cite{MM5}, \cite{MM6}.
 
 One of the main developments established in the above listed papers is the introduction of the notion of a simple transitive $2$-representation, which is a $2$-analogue of the classical notion of a simple module. 
 Similarly to the Jordan-H{\" o}lder theorem for simple modules, simple transitive $2$-representations admit a weak Jordan-H{\" o}lder theory.
 This immediately leads to the very natural problem of classifying such $2$-representations, up to (weak) isomorphism. 
 
 One of the first such classification results was given in \cite{MM5}, stating that if $A$ is a self-injective finite dimensional algebra, then the simple transitive $2$-representations of the $2$-category $\csym{C}_{A}$ of projective bimodules over $A$ are exhausted by the so-called cell $2$-representations. 
 Cell $2$-representations are obtained from the cell structure of $\csym{C}_{A}$, which is defined similarly to the cell structure on a semigroup resulting from Green's relations, and can be defined for any finitary $2$-category.
 This classification result was later significantly improved in \cite{MMZ}, where it was shown that one does not need to assume $A$ to be self-injective. 
 
 The $2$-category $\csym{C}_{A}$ is biequivalent to the delooping of the $\Bbbk$-linear monoidal category $\mathcal{C}_{A} = (\on{add}\setj{A, A \otimes_{\Bbbk} A}, \otimes_{A})$. As such, it has a unique object, its $1$-morphisms correspond to $A$-$A$-bimodules in $\on{add}\setj{A, A \otimes_{\Bbbk} A}$, with composition corresponding to the tensor product over $A$, and $2$-morphisms correspond to bimodule homomorphisms. 
 Given a fixed \idem\ $1 = e_{0} + \cdots + e_{k}$ for $A$ and
 subsets $U,V$ of $\setj{e_{0},\ldots, e_{k}}$, it is easy to verify that $\on{add}\setj{A, Ae \otimes fA \; | e \in U, f \in V}$ gives a monoidal subcategory of $\mathcal{C}_{A}$, and hence gives rise to a respective $2$-subcategory $\csym{D}_{U\times V}$ of $\csym{C}_{A}$. 
 If $V = \setj{e_{0},\ldots, e_{k}}$, then $\csym{D}_{U\times V}$ is given by a union of right cells of $\csym{C}_{A}$, and if $U = \setj{e_{0}, \ldots, e_{k}}$, then $\csym{D}_{U\times V}$ is given by a union of left cells.

 The cell structure of $\csym{D}_{U \times V}$ can be understood in terms of that of $\csym{C}_{A}$. However, whenever $U$ and $V$ do not coincide, the crucial symmetry between the left cells and right cells of $\csym{C}_{A}$ breaks upon restriction to $\csym{D}_{U \times V}$.
 In particular, if $U \neq V$, then $\csym{D}_{U\times V}$ is not weakly fiat, so if one is to study its simple transitive $2$-representations, the rich theory of $2$-representations of weakly fiat $2$-categories developed in, among others, \cite{MMMT} and \cite{MMMTZ1}, cannot be applied directly.
 
 When studying simple transitive $2$-representations of $\csym{D}_{U \times V}$, one should expect some of the techniques of \cite{MM5} (and subsequent papers, such as \cite{MZ1}, \cite{MZ2}) to work also in the setting of the $2$-subcategories. At the same time, the lack of cell symmetry may result in the existence of non-cell simple transitive $2$-representations. 
 If that is the case, one may hope that it would be relatively easy to construct non-cell $2$-representations of such $2$-subcategories, and that the constructions used to that end would then generalize to other, more difficult settings.
 
 A special case of the above described problem was studied in \cite{Zi}. Let $A$ be the zigzag algebra on a star-shaped graph $\Gamma$ on $k+1$ vertices $\setj{0,1,\ldots, k}$, with the unique internal node labelled by $0$. 
 Let $U = \setj{e_{0},e_{1},\ldots, e_{k}}$ and  $V = \setj{e_{0}}$. In this case, the existence of non-cell simple transitive $2$-representations of $\csym{D}_{U \times V}$ was conjectured in \cite{Zi}, and later positively verified in \cite{St}.
 In contrast to this,
 \cite[Theorem~6.2]{Zi} shows that if we instead choose $U = \setj{e_{0}}$ and $V = \setj{e_{0},\ldots e_{k}}$, then any simple transitive $2$-representation is equivalent to a cell $2$-representation.
 
 In order to generalize \cite[Theorem~6.2]{Zi}, we avoid the explicit calculations used in the proof given in \cite{Zi}, and give a clearer connection between that proof and the fact that the zigzag algebra is self-injective. To that end, we observe that
 we only need to use the self-injective property when speaking of the idempotents in $U$, which motivates the main definition of the document:
 \begin{definition*}[Definition \ref{MainDefinition}]
  A subset $U \subseteq \setj{e_{0},\ldots, e_{k}}$ is a {\it self-injective core} if, for every $e \in U$, there is $f \in U$ such that $Ae \simeq (fA)^{\ast}$.
 \end{definition*}

 Our main result is the following:
 \begin{theorem*}[Theorem \ref{MainResult}]
  Let $U \subseteq V \subseteq \setj{e_{0},e_{1},\ldots, e_{k}}$, with $U$ a self-injective core for $A$. Then any simple transitive $2$-representation of $\csym{D}_{U,V}$ is equivalent to a cell $2$-representation.
 \end{theorem*}

 In particular, the above theorem classifies simple transitive $2$-representations whenever $V = \setj{e_{0},e_{1},\ldots, e_{k}}$ and $A$ is self-injective, which implies \cite[Theorem~6.2]{Zi}. The generalization is vast: from a fixed choice of $A$ together with fixed choices of $U$ and $V$, we arrive at a general condition for $A$ and $U$, with an abundance of instances going beyond the setting of \cite{Zi}.
 
 The paper is structured as follows. Section \ref{s2} contains the necessary preliminaries for $2$-representation theory. In Section \ref{s3}, we give a more detailed account of the $2$-categories of the form $\csym{D}_{U\times V}$, including their cell structure and the structure of their cell $2$-representations. Section \ref{s4} consists of the proof of the main theorem. In Section~\ref{s5}, we describe an example illustrating how $2$-categories of the form $\csym{D}_{U \times V}$ appear naturally in the context of general fiat $2$-categories.
 ~\\
 
 \noindent \textbf{Acknowledgments.} This research is partially supported by G{\" o}ran Gustafssons Stiftelse. The author would like to thank the referee for insightful comments, and his advisor, Volodymyr Mazorchuk, for many helpful discussions.

\section{Preliminaries}\label{s2}

Throughout, we let $\Bbbk$ be an algebraically closed field. By $(-)^{\ast}$ we denote the duality functor $\on{Hom}_{\Bbbk}(-,\Bbbk)$.
Given integers $m,n$ with $m < n$, we let $\rr{m,n}$ denote the set $\setj{m,m+1,\ldots,n}$. Further, let $\rr{n} := \rr{1,n}$.

\subsection{\texorpdfstring{ $2$-categories and their $2$-representations}{ 2-categories and their 2-representations}}\label{s1.1}

 We say that a category $\mathcal{C}$ is {\it finitary over $\Bbbk$} if it is additive, idempotent split, $\Bbbk$-linear, and has finitely many isomorphism classes of indecomposable objects. 

 \begin{definition}
 A $2$-category $\csym{C}$ is {\it finitary over $\Bbbk$} if
 \begin{itemize}
  \item it has finitely many objects;
  \item for every $\mathtt{i,j} \in \csym{C}$, the category $\csym{C}(\mathtt{i,j})$ is finitary;
  \item horizontal composition in $\csym{C}$ is $\Bbbk$-bilinear;
  \item for any object $\mathtt{i} \in \csym{C}$, the identity $1$-morphism $\mathbb{1}_{\mathtt{i}}$ is an indecomposable object of $\csym{C}(\mathtt{i,i})$.
 \end{itemize}
 \end{definition}
 For the remainder of this section, let $\csym{C}$ be a $2$-category which is finitary over $\Bbbk$.
Following \cite{MM2}, we consider the following $2$-categories:
\begin{itemize}
 \item $\mathfrak{A}_{\Bbbk}$ - the $2$-category whose objects are small, finitary (over $\Bbbk$) categories, \mbox{$1$-morphisms} are additive $\Bbbk$-linear functors between such categories and whose $2$-morphisms are all natural transformations between such functors;
 \item $\mathfrak{R}_{\Bbbk}$ - the $2$-category whose objects are small abelian $\Bbbk$-linear categories, \mbox{$1$-morphisms} are right exact $\Bbbk$-linear functors between such categories and whose $2$-morphisms are all natural transformations between such functors.
\end{itemize}

\begin{definition}
 A {\it finitary $2$-representation} of $\csym{C}$ is a $2$-functor $\mathbf{M}$ from $\csym{C}$ to $\mathfrak{A}_{\Bbbk}$, such that $\mathbf{M}_{\mathtt{i,j}}$ is $\Bbbk$-linear, for all $\mathtt{i,j} \in \on{Ob}\csym{C}$. 
An {\it abelian $2$-representation} of $\csym{C}$ is such a $2$-functor whose codomain is $\mathfrak{R}_{\Bbbk}$.
\end{definition}

Together with strong transformations and modifications, the collection of finitary $2$-representations of $\csym{C}$ forms a $2$-category $\csym{C}\!\on{-afmod}$. 
Similarly, we obtain the $2$-category $\csym{C}\!\on{-mod}$ of abelian $2$-representations of $\csym{C}$, see \cite[Section~2.3]{MM3} for details.
We say that two $2$-representations $\mathbf{M},\mathbf{N}$ are {\it equivalent} if there exists an invertible strong transformation $\mathbf{M} \rightarrow \mathbf{N}$.
In particular, a strong transformation $\Phi \in \csym{C}\!\on{-afmod}(\mathbf{M},\mathbf{N})$ such that all of its components are equivalences of categories, is invertible, as shown for instance in \cite[Proposition 2]{MM3}.

A {\it $\csym{C}$-stable ideal} $\mathbf{I}$ is a tuple $(\mathbf{I}(\mathtt{i}))_{\mathtt{i} \in \on{Ob}\ccf{C}}$ of ideals $\mathbf{I}(\mathtt{i})$ of $\mathbf{M}(\mathtt{i})$, such that, for any morphism $X \xrightarrow{f} Y$ in $\mathbf{I}(\mathtt{i})$ and any $1$-morphism $\mathrm{F} \in \on{Ob}\csym{C}(\mathtt{i,j})$, the morphism $\mathbf{M}\mathrm{F}(f)$ lies in $\mathbf{I}(\mathtt{j})$.

\begin{definition}
A $2$-representation is said to be {\it transitive} if, for any indecomposable $X \in \on{Ob}\mathbf{M}(\mathtt{j})$ and $Y \in \on{Ob}\mathbf{M}(\mathtt{k})$, there is a $1$-morphism $\mathrm{G}$ in $\csym{C}(\mathtt{j,k})$ such that $Y$ is isomorphic to a direct summand of $\mathbf{M}(\mathrm{F})X$. 
A $2$-representation is {\it simple transitive} if it admits no non-trivial $\csym{C}$-stable ideals. 
\end{definition}

As shown in \cite[Lemma~4]{MM5}, every transitive $2$-representation admits a unique simple transitive quotient.
Further, from \cite[Section~4.2]{MM5} it follows that a simple transitive $2$-representation is transitive.

 If $\csym{C}$ has a unique object $\mathtt{i}$, we define the {\it rank} of a finitary $2$-representation $\mathbf{M}$ of $\csym{C}$ as the number of isoclasses of indecomposable objects of $\mathbf{M}(\mathtt{i})$. If $\csym{C}$ has multiple objects, one can define the rank as a tuple of positive integers. 

We say that $\csym{C}$ is {\it weakly fiat} if every $1$-morphism of $\csym{C}$ admits both left and right adjoints, giving rise to a weak antiautomorphism $(-)^{\ast}$ of finite order, see \cite[Section~2.5]{MM6}. If $(-)^{\ast}$ is involutive, we say that $\csym{C}$ is {\it fiat}. 

\subsection{Abelianization}

Given a finitary $2$-representation $\mathbf{M}$ of $\csym{C}$, let $\overline{\mathbf{M}}$ denote the {\it projective abelianization} of $\mathbf{M}$, as initially defined in \cite{MM2}, using the projective abelianization of finitary categories given in \cite{Fr}.

The abelianization is an abelian $2$-representation of $\csym{C}$ associated to $\mathbf{M}$, and, more generally, we obtain a $2$-functor $\overline{\,\,\cdot\,\,}: \csym{C}\!\on{-afmod} \rightarrow \csym{C}\!\on{-mod}$.
The finitary $2$-repre\-senta\-tion $\mathbf{M}$ can be recovered from $\overline{\mathbf{M}}$ by restriction to certain subcategories of $\overline{\mathbf{M}}(\mathtt{i})$ equivalent to $\overline{\mathbf{M}}(\mathtt{i})\!\on{-proj}$. If $\csym{C}$ is weakly fiat, then, since $\overline{\mathbf{M}}$ is abelian, 
the functor $\overline{\mathbf{M}}\mathrm{F}$ is exact, for any $1$-morphism $\mathrm{F}$ of $\csym{C}$.
An improved construction was given in \cite[Section~3]{MMMT}. For our purposes, either of the two constructions can be used.

As observed in \cite[Section~3.1]{MM1}, there are finite dimensional algebras $A_{\mathtt{i}}$ such that we have $\mathbf{M}(\mathtt{i})\simeq A_{\mathtt{i}}\!\on{-proj}$ and $\overline{\mathbf{M}}(\mathtt{i})\simeq A_{\mathtt{i}}\!\on{-mod}$, for all $\mathtt{i} \in \on{Ob}\csym{C}$. 
Since, for a $1$-morphism $\mathrm{F}$ of $\csym{C}(\mathtt{i,j})$, the functor $\overline{\mathbf{M}}\mathrm{F}$ is right-exact, it follows by the Eilenberg-Watts theorem that there is a finitely generated $A_{\mathtt{j}}$-$A_{\mathtt{i}}$-bimodule $F$ such that 
$\mathbf{M}\mathrm{F}$ corresponds to $F \otimes_{A_{\mathtt{i}}} - $.
Similarly, for a $2$-morphism $\alpha: \mathrm{F} \rightarrow \mathrm{G}$, the natural transformation $\mathbf{M}\alpha$ can be identified with a bimodule homomorphism.

\subsection{Cells and cell 2-representations}\label{CellSection}

Let $\mathcal{S}(\csym{C})$ denote the set of isomorphism classes of indecomposable $1$-morphisms of $\csym{C}$. Given an object $X$ of a category $\mathcal{C}$, we denote its isomorphism class by $[X]$.
If $\mathtt{K}$ is a set of isomorphism classes of objects in a category $\mathcal{C}$, and $X$ is an object of $\mathcal{C}$, we will sometimes abuse notation and write $X \in \mathtt{K}$ for $[X] \in \mathtt{K}$.
In particular, if $\mathrm{F}$ is an indecomposable $1$-morphism of $\csym{C}$, we may write $\mathrm{F} \in \mathcal{S}(\csym{C})$.

Following \cite{MM2}, given $\mathrm{F},\mathrm{G} \in \mathcal{S}(\csym{C})$, we write $\mathrm{F} \geq_{L} \mathrm{G}$ if there is a $1$-morphism $\mathrm{H}$ such that $\mathrm{F}$ is isomorphic to a direct summand of $\mathrm{H} \circ \mathrm{G}$. This gives the {\it left preorder} $L$ on $\mathcal{S}(\csym{C})$. The {\it right preorder} $R$ and the {\it two-sided preorder} $J$ are defined similarly. The equivalence classes of the induced equivalence relations are called the left, right and two-sided {\it cells} respectively. (Alternatively, $L$-cells, $R$-cells and $J$-cells.)
 
Let $\mathcal{L}$ be a left cell of $\csym{C}$. We briefly summarize the construction of the {\it cell $2$-representation $\mathbf{C}_{\mathcal{L}}$}, following \cite[Section~3.3]{MM5}. There is a unique $\mathtt{i} \in \on{Ob}\csym{C}$ such that $\mathrm{F} \in \mathcal{L}$ implies that $\mathtt{i}$ is the domain of $\mathrm{F}$.
 Consider the $2$-functor $\mathbf{P}_{\mathtt{i}} := \csym{C}(\mathtt{i},-)$. This $2$-functor gives a finitary $2$-representation of $\csym{C}$, and the additive closure of the set of $1$-morphisms $\mathrm{F}$ satisfying $\mathrm{F} \geq_{L} \mathcal{L}$ gives a $2$-subrepresentation $\mathbf{K}_{\mathcal{L}}$ of $\mathbf{P}_{\mathtt{i}}$. 
 Taking the quotient of the latter by the ideal generated by $\setj{\on{id}_{\mathrm{F}} \; | \; \mathrm{F} >_{L} \mathcal{L}}$ gives a transitive $2$-representation $\mathbf{N}_{\mathcal{L}}$.
 The cell $2$-representation $\mathbf{C}_{\mathcal{L}}$ is the unique simple transitive quotient of $\mathbf{N}_{\mathcal{L}}$.
 
 A two-sided cell $\mathcal{J}$ of $\csym{C}$ is said to be {\it idempotent} if there are non-zero $1$-morphisms $\mathrm{F},\mathrm{G},\mathrm{H} \in \mathcal{J}$ such that $\mathrm{H}$ is isomorphic to a direct summand of $\mathrm{G} \circ \mathrm{F}$. 
 As was shown in \cite[Lemma~3]{CM}, the set of $J$-cells of $\csym{C}$ not annihilating a fixed transitive $2$-representation $\mathbf{M}$ admits a $J$-greatest element, called the {\it apex of $\mathbf{M}$}.
 The apex of a transitive $2$-representation must be idempotent, and it coincides with the apex of its unique simple transitive quotient.
 
 We now formulate and give a proof of a statement often implicitly used in the literature, for instance in \cite[Proposition~9]{MM5}, \cite{MZ1}, \cite{Zi}. 
 
 Assume that $\csym{C}$ has a single object $\mathtt{i}$, and that $\csym{C}$ admits an idempotent $J$-cell $\mathcal{J}$.
 Let $\mathbf{M}$ be a simple transitive $2$-representation with apex $\mathcal{J}$. 
 Fix a complete, irredundant set $\mathtt{X} = \setj{X_{1},\ldots, X_{n}}$ of representatives of isomorphism classes of indecomposable objects of $\mathbf{M}(\mathtt{i})$. By construction of projective abelianization, $\bigoplus_{k=1}^{n}X_{k}$ is a projective generator for $\overline{\mathbf{M}}(\mathtt{i}) \simeq A_{\mathtt{i}}\!\on{-mod}$. Let $Q:= A_{\mathtt{i}}$. We conclude that
 $
  \mathbf{M}(\mathtt{i}) \simeq Q\!\on{-proj},
 $
 and $Q = \on{End}(\bigoplus_{k=1}^{n}X_{k})^{\on{op}}$.
 Further, the set $\setj{f_{k}}_{k=1}^{n} :=\setj{\on{id}_{X_{k}}}_{k=1}^{n}$ is a \idem\ for $Q$. Choose a left cell $\mathcal{L} = \setj{\mathrm{F}_{1},\ldots,\mathrm{F}_{r}}$ in $\mathcal{J}$. Finally, recall that the {\it Cartan matrix of $\mathbf{M}(\mathtt{i})$ with respect to $\mathtt{X}$} (and similarly for any finitary category with a fixed list of indecomposables) is the $n \times n$ matrix whose $(i,j)$th entry is given by $\on{dim}\on{Hom}_{\mathbf{M}(\mathtt{i})}(X_{i},X_{j})$.
 \begin{lemma}\label{StdArg}
 If there is an ordering of $\mathtt{X}$ such that
 \begin{itemize}
  \item The Cartan matrix of $\mathbf{M}(\mathtt{i})$ is equal to that of $\on{add}(\mathcal{L})$;
  \item There is an index $l \in \rr{n}$ such that, for $k \in \rr{n}$, the functor $\mathbf{M}\mathrm{F}_{k}$ is naturally isomorphic to the functor $Qf_{k} \otimes_{\Bbbk} f_{l}Q \otimes_{Q} -$.
 \end{itemize}
  then $\mathbf{M}$ is equivalent to the cell $2$-representation $\mathbf{C}_{\mathcal{L}}$.
\end{lemma}

\begin{proof}
 As described in \cite[Section~5.2]{MM5}, the embedding $\mathbf{M}(\mathtt{i}) \rightarrow \overline{\mathbf{M}}(\mathtt{i})$, sending $X$ to the projective object $0 \rightarrow X$, gives an equivalence $\mathbf{M}(\mathtt{i}) \simeq \overline{\mathbf{M}}(\mathtt{i})\!\on{-proj}$.
 From the latter of our assumptions, we conclude that for any $1$-morphism $\mathrm{F}$ of $\csym{C}$ and any $X \in \on{Ob}\overline{\mathbf{M}}(\mathtt{i})$, the object $\mathbf{M}\mathrm{F}X$ is projective. 
 
 In particular, $\overline{\mathbf{M}}(\mathtt{i})\!\on{-proj}$ becomes a finitary $2$-representation of $\csym{C}$, equivalent to $\mathbf{M}$, which we denote by $\overline{\mathbf{M}}\!\on{-proj}$. 
 Given $k \in \rr{n}$, we let $L_{k}$ be the simple top of the object $0 \rightarrow X_{k}$ of $\overline{\mathbf{M}}(\mathtt{i})$. 
 
 The Yoneda lemma for $\mathbf{P}_{\mathtt{i}}$ given in \cite[Lemma~9]{MM2} gives a unique $2$-transformation from $\mathbf{P}_{\mathtt{i}}$ to $\overline{\mathbf{M}}\!\on{-proj}$, induced by the assignment $\mathbb{1}_{\mathtt{i}} \mapsto L_{l}$. This $2$-transformation sends $\mathrm{F}$ to $\overline{\mathbf{M}}\mathrm{F} L_{l}$.
 We may restrict this to a $2$-transformation from $\mathbf{K}_{\mathcal{L}}$.
 Since $\mathcal{J}$ is idempotent, \cite[Corollary~19]{KM} implies that the left cells it contains are pairwise $L$-incomparable, and so $\mathrm{F} >_{L} \mathcal{L}$ implies $\mathrm{F} >_{J} \mathcal{J}$. Since $\mathcal{J}$ is the apex of $\mathbf{M}$, a $1$-morphism $\mathrm{F}$ satisfying $\mathrm{F} >_{L} \mathcal{L}$ is sent to zero by the above described $2$-transformation. 
 Thus, the above $2$-transformation from $\mathbf{K}_{\mathcal{L}}$ factors through an induced $2$-transformation $\sigma: \mathbf{N}_{\mathcal{L}} \rightarrow \overline{\mathbf{M}}\!\on{-proj}$, from the transitive quotient $\mathbf{N}_{\mathcal{L}}$ of $\mathbf{K}_{\mathcal{L}}$. 
 
 The image of $\mathrm{F}_{k}$ under 
 $\sigma$ is 
 \[
 \overline{\mathbf{M}}\mathrm{F}_{k}L_{l} \simeq Qf_{k} \otimes_{\Bbbk} f_{l}Q \otimes_{Q} L_{l} \simeq Qf_{k},
 \]
 which is indecomposable. This shows that all the isomorphism classes of indecomposable objects of $Q\!\on{-proj}$ are in the essential image of $\sigma$, and so $\sigma$ is essentially surjective.
 
 The kernel of $\sigma$ is an ideal of $\mathbf{N}_{\mathcal{L}}$, which does not contain any identity $2$-morphisms of $\csym{D}$, since $\mathrm{F}_{k}L_{l} \neq 0$, for all $k$. Thus it is contained in the maximal ideal $\mathbf{I}$ of $\mathbf{N}_{\mathcal{L}}$, which defines the cell $2$-representation $\mathbf{C}_{\mathcal{L}}$, via $\mathbf{C}_{\mathcal{L}} = \mathbf{N}_{\mathcal{L}}/\mathbf{I}$. 
 
 We claim that $\on{Ker}\sigma = \mathbf{I}$. It suffices to show that $\mathbf{I} \subseteq \on{Ker}\sigma$. The $2$-transformation
 \[ \widetilde{\sigma}: \mathbf{N}_{\mathcal{L}}/\on{Ker}(\sigma) \rightarrow \overline{\mathbf{M}}\!\on{-proj} \]
 is, by the definition of $\on{Ker}\sigma$, given by a faithful functor, so that, for all $s,t \in \rr{n}$, the linear map
 \[ \widetilde{\sigma}_{st}: \on{Hom}_{\mathbf{N}_{\mathcal{L}}/\on{Ker}(\sigma)}(F_{s}, F_{t}) \rightarrow \on{Hom}_{Q\!\on{-proj}}(Qe_{s}, Qe_{t}) \]
 is injective. Hence
 \[
 \on{dim}\on{Hom}_{\mathbf{N}_{\mathcal{L}}/\on{Ker}(\sigma)}(\mathrm{F}_{s}, \mathrm{F}_{t}) \leq \on{dim}\on{Hom}_{Q\!\on{-proj}}(Qe_{s}, Qe_{t}).
 \]
 Since $\on{Ker}(\sigma) \subseteq \mathbf{I}$, we have
 \[ \on{dim}\on{Hom}_{\mathbf{C}_{\mathcal{L}}}(\mathrm{F}_{s},\mathrm{F}_{t}) \leq \on{dim}\on{Hom}_{\mathbf{N}_{\mathcal{L}}/\on{Ker}(\sigma)}(\mathrm{F}_{s}, \mathrm{F}_{t}). \]
 The equality of Cartan matrices for $\mathbf{C}_{\mathcal{L}}$ and $\mathbf{M}$ implies that the lower and the upper bounds for $\on{dim}\on{Hom}_{\mathbf{N}_{\mathcal{L}}/\on{Ker}(\sigma)}(\mathrm{F}_{s}, \mathrm{F}_{t})$ coincide, so 
 \[
\on{dim}\mathbf{I}(\mathrm{F}_{s},\mathrm{F}_{t}) = \on{dim}\on{Ker}(\sigma)(\mathrm{F}_{s},\mathrm{F}_{t}), \text{ and } \on{Ker}(\sigma) = \mathbf{I}. 
 \]
We thus have a $2$-transformation
 $
  \widetilde{\sigma}: \mathbf{C}_{\mathcal{L}} \rightarrow \overline{\mathbf{M}}\!\on{-proj}
 $
 such that $\widetilde{\sigma}_{st}$ is an injective linear map between equidimensional spaces, and thus is an isomorphism, for all $s,t$. This shows that $\widetilde{\sigma}$ is essentially surjective, full and faithful, and so it gives an equivalence of $2$-representations. Thus, $\mathbf{C}_{\mathcal{L}}$ is equivalent to $\overline{\mathbf{M}}\!\on{-proj}$, and hence also to $\mathbf{M}$.
\end{proof}

 \subsection{Decategorification and action matrices}
 Following \cite[Section~2.4]{MM2}, we define the {\it decategorification of $\csym{C}$} as the preadditive category $[\csym{C}]$ given by
 \begin{itemize}
  \item $\on{Ob}[\csym{C}] := \on{Ob}\csym{C}$;
  \item Given $\mathtt{i,j} \in \on{Ob}\csym{C}$, we let $[\csym{C}](\mathtt{i,j}) := [\csym{C}(\mathtt{i,j})]_{\oplus}$, the split Grothendieck group of $\csym{C}(\mathtt{i,j})$, with composition induced by composition in $\csym{C}$.
 \end{itemize}
 Similarly, for a finitary $2$-representation $\mathbf{M}$, its decategorification is a $\mathbb{Z}$-bilinear functor from $[\csym{C}]$ to $\mathbf{Ab}$. 
 
 Given $\mathtt{i,j} \in \on{Ob}\csym{C}$, choose complete, irredundant sets of representatives of isoclasses of indecomposable objects of the categories $\mathbf{M}(\mathtt{i})$ and $\mathbf{M}(\mathtt{j})$ and denote them by $\setj{X_{1},\ldots, X_{n}}$ and $\setj{ Y_{1},\ldots, Y_{m}}$, respectively. Let $\mathrm{F} \in \on{Ob}\csym{C}(\mathtt{i,j})$.
 With respect to the induced bases $\mathtt{X,Y}$ in the respective split Grothendieck groups, the {\it action matrix} $[\mathrm{F}]^{\mathtt{X,Y}}$ is the $m \times n$ matrix such that the entry $[\mathrm{F}]_{kl}^{\mathtt{X},\mathtt{Y}}$ is the multiplicity of $Y_{k}$ as a direct summand of $\mathbf{M}\mathrm{F}X_{l}$.
 Under a fixed choice of bases, we will denote $[\mathrm{F}]^{\mathtt{X,Y}}$ simply by $[\mathrm{F}]$. If there is possible ambiguity regarding the $2$-representation with respect to which we denote the action matrix, we will add it as a subscript in our notation, for instance $[\mathrm{F}]_{\mathbf{M}}$ in the case above.
 
\subsection{\texorpdfstring{ Discrete extensions of $2$-representations}{Discrete extensions of 2-representations}}

The notion of a discrete extension of $2$-representations was introduced and studied in \cite{CM}.
We only give a brief summary of the notions we will use in this text and refer to \cite{CM} for the details.

A {\it short exact sequence}
\begin{equation}\label{SES}
 0 \rightarrow \mathbf{K} \rightarrow \mathbf{M} \rightarrow \mathbf{N} \rightarrow 0
\end{equation}
of $2$-representations 
of $\csym{C}$ 
consists of
\begin{itemize}
 \item a finitary $2$-representation $\mathbf{M}$;
 \item a finitary $2$-subrepresentation $\mathbf{K}$ of $\mathbf{M}$;
 \item the $2$-representation $\mathbf{N}$ given by the quotients of the categories $\mathbf{M}(\mathtt{i})$ by the ideals generated by the identity morphisms of all objects of $\mathbf{K}(\mathtt{i})$;
 \item the inclusion $2$-transformation $\mathbf{K} \rightarrow \mathbf{M}$;
 \item the projection $2$-transformation $\mathbf{M} \rightarrow \mathbf{N}$.
\end{itemize}

Given $\mathtt{i} \in \on{Ob}\csym{C}$, a short exact sequence such as the one above induces a partition of the set of indecomposable objects of $\mathbf{M}(\mathtt{i})$ into two subsets, one consisting of the objects of $\mathbf{K}(\mathtt{i})$ and the other consisting of the indecomposable objects not sent to zero under the projection $\mathbf{M} \rightarrow \mathbf{N}$. Let $\setj{Y_{1},\ldots, Y_{s}}$ be the former set of indecomposable objects and let $\setj{X_{1},\ldots, X_{r}}$ be the latter. 

We say that the short exact sequence \eqref{SES} is {\it trivial} if, for any $1$-morphism $\mathrm{F}$ of $\csym{C}$, there are no $l \in \rr{s}, k \in \rr{r}$, such that $Y_{l}$ is a direct summand of $\mathbf{M}\mathrm{F}X_{k}$.

Let $\mathbf{K,N}$ be transitive, finitary $2$-representations of $\csym{C}$. If every short exact sequence
\[
 0 \rightarrow \mathbf{K'} \rightarrow \mathbf{M} \rightarrow \mathbf{N'} \rightarrow 0
\]
 with $\mathbf{K} \simeq \mathbf{K'}$ and $\mathbf{N} \simeq \mathbf{N'}$ is trivial, we have
 \[
  \on{Dext}(\mathbf{N},\mathbf{K}) = \varnothing.
 \]
 We will not use the general definition of $\on{Dext}(\mathbf{N},\mathbf{K})$, and refer the interested reader to \cite{CM}.

\subsection{The \texorpdfstring{$2$-category $\csym{C}_{\!A}$}{2-category C(A)}}\label{CAFacts}

 Let $A$ be a finite dimensional, basic, connected algebra. Fix a small category $\mathcal{A}$ equivalent to $A\!\on{-mod}$. The $2$-category $\csym{C}_{\!A}$ consists of
 \begin{itemize}
  \item a single object $\mathtt{i}$;
  \item endofunctors of $\mathcal{A}$ isomorphic to tensoring with $A$-$A$-bimodules in the category $\on{add}\setj{A \otimes_{\Bbbk} A, A}$ as $1$-morphisms (in other words, so-called projective functors of $\mathcal{A}$); 
  \item all natural transformations between such functors as $2$-morphisms.
 \end{itemize}

 In particular, $\csym{C}_{\!A}$ is finitary. 
 Fix a \idem\ $\setj{e_{1}, \ldots, e_{m}}$ and an equivalence $A\!\on{-mod}\xiso \mathcal{A}$. Under these choices, the isomorphism classes of indecomposable $1$-morphisms correspond bijectively to the set consisting of the regular bimodule ${}_{A}A_{A}$ and the projective bimodules $\setj{Ae_{i} \otimes_{\Bbbk} e_{j}A \; | \; i,j \in \rr{m}}$. 
 We then further choose a unique representative of every isomorphism class corresponding to a bimodule of the form $Ae_{i} \otimes e_{j}A$, and denote this representative by $\mathrm{F}_{ij}$. On the level of isomorphism classes of $1$-morphisms, the composition $\mathrm{F}_{ij} \circ \mathrm{F}_{kl}$ corresponds to the tensor product
\[
 (Ae_{i} \otimes_{\Bbbk} e_{j}A) \otimes_{A}(Ae_{k} \otimes_{\Bbbk} e_{l}A) \simeq (Ae_{i} \otimes_{\Bbbk} e_{l}A)^{\oplus \dim e_{j}Ae_{k}},
\]
 and so we may write
\begin{equation}\label{CAComposition}
 \mathrm{F}_{ij} \circ \mathrm{F}_{kl} \simeq \mathrm{F}_{il}^{\oplus \dim e_{j}Ae_{k}}.
\end{equation}
Using our notation for $1$-morphisms, given a subset $W \subseteq \rr{m}\times \rr{m}$, we denote the set $\setj{[\mathrm{F}_{ij}] \; | \; (i,j) \in W}$ by $\mathcal{S}(W)$. 
A $1$-morphism $\mathrm{F}$ belongs to $\mathcal{S}(\rr{m}\times \rr{m})$ if and only if $\mathrm{F}$ is an indecomposable $1$-morphism not isomorphic to the identity $1$-morphism.

Further, to $W$ as above we associate the multiplicity-free $1$-morphism $\bigoplus_{(i,j) \in W} \mathrm{F}_{ij}$, which we denote by $\mathrm{F}_{W}$. In particular, $\mathrm{F}_{ij} = \mathrm{F}_{\setj{(i,j)}}$.

Unless otherwise stated, we implicitly assume a fixed \idem\ $\setj{e_{1},\ldots, e_{m}}$ for $A$, with respect to which we use the above introduced notation.

A $2$-category of the form $\csym{C}_{\!A}$ is weakly fiat if and only if $A$ is self-injective. This is an immediate consequence of the following lemma:
 \begin{lemma}\label{AdjSelf}\cite[Lemma~45]{MM1}
  Let $f,e$ be primitive, mutually orthogonal idempotents of $A$. Then 
  \[
  \big( \left(Ae \otimes_{\Bbbk} fA\right) \otimes_{A} -, \left( (fA)^{*} \otimes_{\Bbbk} eA \right) \otimes_{A} - \big) 
  \]
  is an adjoint pair of endofunctors of $A\!\on{-mod}$.
 \end{lemma}
 
 \begin{proof}
  This is an immediate consequence of the proof of \cite[Lemma~45]{MM1}, since the assumption about $A$ being weakly symmetric is used there only to show that $(fA)^{*}$ is projective.
 \end{proof}

 \subsubsection{Cell structure and simple transitive \texorpdfstring{$2$-representations of $\csym{C}_{\!A}$}{2-representations of C(A)}}
 Using the composition rule given in \eqref{CAComposition}, one may determine the cell structure of $\csym{C}_{A}$. The following is its representation using a so-called eggbox diagram, commonly used in the theory of semigroups:
\begin{gather*}
 \begin{array}{|c|}
  \hline 
  \mathbb{1}_{\mathtt{i}} \\
  \hline
 \end{array} \\
\begin{array}{|c|c|c|c|}
\hline
\mathrm{F}_{11}&\mathrm{F}_{12}&\cdots &F_{1m}\\
\hline
\mathrm{F}_{21}&\mathrm{F}_{22}&\cdots&F_{2m}\\
\hline
\vdots&\vdots&\ddots &\vdots\\
\hline
\mathrm{F}_{m1}&\mathrm{F}_{m2}&\cdots &\mathrm{F}_{mm}\\
\hline
\end{array}
\end{gather*}
In the above diagram, the rectangular arrays are the $J$-cells of $\csym{C}_{A}$, the rows within the arrays right cells and the columns give left cells.
The diagram thus illustrates that $\csym{C}_{\!A}$ has two $J$-cells $\mathcal{J}_{0},\mathcal{J}_{1}$, satisfying $\mathcal{J}_{1} >_{J} \mathcal{J}_{0}$, where $\mathcal{J}_{1}$ is partitioned into $m$ left cells and $m$ right cells, with each intersection of a left cell with a right cell given by a singleton. By construction we then have $m+1$ cell $2$-representations: the apex of a cell $2$-representation $\mathbf{C}_{\mathcal{L}}$ here is the $J$-cell containing $\mathcal{L}$.
 
 The simple transitive $2$-representations of $\csym{C}_{\!A}$ have been completely classified:
\begin{theorem}\cite[Theorem~9]{MMZ}
\begin{itemize}
 \item Every simple transitive $2$-representation of $\csym{C}_{\!A}$ is equivalent to a cell $2$-representation. 
 \item All cell $2$-representations associated to left cells contained in $\mathcal{J}_{1}$ are mutually equivalent.
 \item Thus, up to equivalence, the $2$-category $\csym{C}_{A}$ admits exactly two simple transitive $2$-representations.
\end{itemize}
\end{theorem}

The main aim of the next two sections is to obtain a similar result for certain $2$-subcategories of $\csym{C}_{\!A}$.

\section{\texorpdfstring{Combinatorial $2$-subcategories of $\csym{C}_{\!A}$}{Combinatorial 2-subcategories of C(A)}}\label{s3}
 A {\it multisemigroup} is a set $\mathtt{S}$ together with a function $\mu$ from $\mathtt{S} \times \mathtt{S}$ to the power set $2^{\mathtt{S}}$, satisfying the {\it associativity condition}
 \[
  \bigcup_{x \in \mu(r,s)} \mu(x,t) = \bigcup_{y \in \mu(s,t)} \mu(r,y).
 \]
 A {\it multisubsemigroup} of $S$ is a subset $R$ of $S$ such that, for any $x,y \in R$, we have $\mu(x,y) \subseteq R$. The restriction of $\mu$ to $R\times R$ endows $R$ with the structure of a multisemigroup.

 Let $\csym{C}$ be a finitary $2$-category. Denote by $\mathcal{S}(\csym{C})^{0}$ the set $\mathcal{S}(\csym{C}) \sqcup \setj{0}$.
 Recall from \cite[3.3]{MM2} that the {\it multisemigroup of $\csym{C}$} is defined as the set $\mathcal{S}(\csym{C})^{0}$ together with the function $\mu$ defined by
 \[
  \mu([\mathrm{F}],[\mathrm{G}]) =
  \begin{cases}
   \setj{0} \text{ if } \mathrm{F} \circ \mathrm{G} \text{ is undefined;} \\
   \setj{0} \text{ if } \mathrm{F} \circ \mathrm{G} = 0 \\
   \setj{[\mathrm{H}] \in \mathcal{S}(\csym{C}) \; | \; \mathrm{H} \text{ is a direct summand of } \mathrm{F}\circ \mathrm{G}} \text{ otherwise.}
  \end{cases}
 \]
 together with $\mu([\mathrm{F}],0) = \mu(0,[\mathrm{F}]) = \mu(0,0) = \setj{0}$.

 A subcategory $\mathcal{D}$ of a category $\mathcal{C}$ is called {\it replete} if, given $X \in \on{Ob}\mathcal{D}$ and an isomorphism $f:X \xrightarrow{\simeq} Y$ in $\mathcal{C}$, the object $Y$ lies in $\on{Ob}\mathcal{D}$, and $f$ is a morphism in $\mathcal{D}$.  Further, $\mathcal{D}$ is called {\it wide} if $\on{Ob}\mathcal{D} = \on{Ob}\mathcal{C}$. 

 If a $2$-subcategory $\csym{D}$ of a $2$-category $\csym{C}$ is such that, for all $\mathtt{i,j} \in \on{Ob}\csym{D}$, the subcategory $\csym{D}(\mathtt{i,j})$ of $\csym{C}(\mathtt{i,j})$ is replete, we say that $\csym{D}$ is {\it $2$-replete.}
 Similarly, we say that a $2$-subcategory $\csym{D}$ of a $2$-category $\csym{C}$ is {\it wide} if $\on{Ob}\csym{D} = \on{Ob}\csym{C}$.

In the remainder of this text we will only consider finitary, wide, $2$-replete, $2$-full $2$-subcategories of finitary $2$-categories. Hence we introduce the following terminology:
\begin{definition}
 Let $\csym{C}$ be a finitary $2$-category. We say that a $2$-subcategory $\csym{D}$ of $\csym{C}$ is {\it combinatorial} if it is finitary, wide, $2$-replete and $2$-full.
\end{definition}

\begin{proposition}\label{CombinatorialSubcategories}
 The map
 \begin{align*}
  \setj{
  \begin{aligned}
   &\text{Combinatorial } \\
   &2\text{-subcategories of } \csym{C} 
  \end{aligned}
  }
  &\rightarrow
  \setj{
  \begin{aligned}
  &\text{Multisubsemigroups of }\mathcal{S}(\csym{C})^{0} \\
  &\text{ containing }0 \text{ and } [\mathbb{1}_{\mathtt{i}}] \text{ for every } \mathtt{i} \in \on{Ob}\csym{C}
  \end{aligned}
  }\\
  \csym{D} &\mapsto \mathcal{S}(\csym{D})^{0} 
\end{align*}
is a bijection.
\end{proposition}

\begin{proof}
By definition, $\mathcal{S}(\csym{D})^{0}$ contains $0$, and, since $\csym{D}$ is wide, $\mathcal{S}(\csym{D})^{0}$ necessarily contains $[\mathbb{1}_{\mathtt{i}}]$, for every $\mathtt{i} \in \on{Ob}\csym{C}$. 
If $\mathrm{F}$ is a $1$-morphism in $\csym{D}(\mathtt{i,j})$ and $G$ a summand of $\mathrm{F}$ in $\csym{C}(\mathtt{i,j})$, then $\csym{D}(\mathtt{i,j})$, being idempotent split, contains a $1$-morphism isomorphic to $\mathrm{G}$, and thus, since $\csym{D}$ is $2$-replete, it also contains $\mathrm{G}$ itself. Hence the map in the proposition is well-defined.
 
 The set $\mathcal{S}(\csym{D})^{0}$ uniquely determines the collection of $1$-morphisms of $\csym{D}$ - the latter is given by all $1$-morphisms isomorphic to finite direct sums of $1$-morphisms in the isomorphism classes of $\mathcal{S}(\csym{D})$. Further, $2$-fullness implies that the collection of $2$-morphisms in $\csym{D}$ is determined by that of $1$-morphisms in $\csym{D}$. Finally, being wide implies $\on{Ob}\csym{D} = \on{Ob}\csym{C}$. It follows that $\mathcal{S}(\csym{D})^{0}$ uniquely determines $\csym{D}$, so the map is injective.
 
 Finally, we show that the map is surjective, since given a multisubsemigroup $\mathtt{S}$ of $\mathcal{S}(\csym{C})^{0}$, the set $\on{Ob}\csym{C}$, together with the collection of all $1$-morphisms isomorphic to finite direct sums of $1$-morphisms in the classes of $\mathtt{S}$ and the collection of all $2$-morphisms of $\csym{C}$ between such $1$-morphisms, gives a wide, $2$-replete, $2$-full, finitary $2$-subcategory $\csym{D}_{\mathtt{S}}$ of $\csym{C}$. In particular, it is closed under composition of $1$-morphisms: given indecomposable $1$-morphisms $\mathrm{F},\mathrm{G}$ of $\csym{D}_{\mathtt{S}}$, all the isomorphism classes of indecomposable direct summands of $\mathrm{G} \circ \mathrm{F}$ again lie in $\mathtt{S}$ and so $\mathrm{G} \circ \mathrm{F}$ is a $1$-morphism of $\csym{D}_{\mathtt{S}}$.
\end{proof}

From the description of the $1$-morphisms of $\csym{C}_{\!A}$ in Section \ref{CAFacts}, one sees that the multisemigroup of $\csym{C}_{\!A}$ is actually a monoid with zero, with multiplication $\mu$ defined by
\[
 \mu([\mathrm{F}_{ij}],[\mathrm{F}_{kl}]) =
 \begin{cases}
 \begin{aligned}
  \setj{[\mathrm{F}_{il}]} &\text{ if } e_{j}Ae_{k} \neq 0\\
  \setj{0} \; \, \, &\text{ otherwise.}
 \end{aligned}
 \end{cases}
\]
We may relabel the element $[\mathbb{1}_{\mathtt{i}}]$ as $1$, and the elements $[\mathrm{F}_{ij}]$ of this monoid as $(i,j)$ and identify the monoid with that obtained by adjoining the unit $1$ to the resulting semigroup structure on $\left(\rr{m} \times \rr{m}\right) \cup \setj{0}$. Denote this monoid by $\mathcal{N}(A)$. Given a submonoid $U$ of $\mathcal{N}(A)$, let $\csym{D}_{U}$ denote the combinatorial $2$-subcategory of $\csym{C}_{A}$ corresponding to it under the bijection of Proposition \ref{CombinatorialSubcategories}.

 \begin{definition}\label{MainDefinition}
  Let $A$ be a finite dimensional, basic, connected algebra and choose a \idem \ $\setj{e_{1},\ldots, e_{m}}$ for $A$. A non-empty subset $U \subseteq \rr{m}$ is a {\it self-injective core for $A$} if, for any $i \in U$, there is $j \in U$ such that
  \[
   (e_{i}A)^{\ast} \simeq Ae_{j},
  \]
 \end{definition}
 
  Let $\csym{D}$ be a combinatorial subcategory of $\csym{C}_{\!A}$, where we have fixed a \idem\ $\setj{e_{1},\ldots, e_{n}}$ for $A$. Define the sets $\mathtt{N}_{L}(\csym{D})$ and $\mathtt{N}_{R}(\csym{D})$ as
  \[
   \begin{aligned}
    &\mathtt{N}_{L}(\csym{D}) := \setj{i \in \rr{m} \; | \; \text{ there is } \mathrm{F}_{kl} \in \on{Ob}\csym{D}(\mathtt{i,i}) \text{ such that } k=i}; \\
    &\mathtt{N}_{R}(\csym{D}) := \setj{i \in \rr{m} \; | \; \text{ there is } \mathrm{F}_{kl} \in \on{Ob}\csym{D}(\mathtt{i,i}) \text{ such that } l=i}.
   \end{aligned}
  \]
 \begin{lemma}\label{LRLemma}
  If a combinatorial subcategory $\csym{D}$ of $\csym{C}_{\!A}$ is weakly fiat, then 
  \[
  \mathtt{N}_{L}(\csym{D}) = \mathtt{N}_{R}(\csym{D}),
  \]
  and this set is a self-injective core for $A$.
 \end{lemma}

 \begin{proof}
  From Lemma \ref{AdjSelf} we conclude that the indecomposable $1$-morphism $\mathrm{F}_{ij}$ of $\csym{D}$ has a right adjoint in $\csym{D}$ if and only if there is $k \in \rr{m}$ such that there is an isomorphism of bimodules
  \[
   (e_{j}A)^{\ast} \otimes_{\Bbbk} e_{i}A \simeq Ae_{k} \otimes_{\Bbbk} e_{i}A
  \]
 and $\mathrm{F}_{ki} \in \on{Ob}\csym{D}(\mathtt{i,i})$. Said bimodules are isomorphic if and only if $(e_{j}A)^{\ast} \simeq Ae_{k}$. In particular, we have shown that $\mathrm{F}_{ij} \in \on{Ob}\csym{D}(\mathtt{i,i})$ implies $\mathrm{F}_{ki} \in \on{Ob}\csym{D}(\mathtt{i,i})$, so $\mathtt{N}_{L}$ is a subset of $\mathtt{N}_{R}$.
 
 As remarked in Section \ref{s1.1}, taking right adjoints gives a weak antiautomorphism of finite order, and so every indecomposable $1$-morphism $\mathrm{F}_{ij}$ of $\csym{D}$ itself is a right adjoint. Thus, by the argument above, there is $l$ such that
 \[
  (e_{l}A)^{\ast} \otimes_{\Bbbk} e_{j}A \simeq Ae_{i} \otimes_{\Bbbk} e_{j}A
 \]
 and $\mathrm{F}_{kl} \in \on{Ob}\csym{D}(\mathtt{i,i})$. In particular, $(e_{l}A)^{\ast} \simeq Ae_{i}$. Similarly to the first part of the proof, this shows that $\mathtt{N}_{R}$ is a subset of $\mathtt{N}_{L}$. Hence $\mathtt{N}_{L} = \mathtt{N}_{R}$.
 
 Given $j \in \mathtt{N}_{R}$, choose $i \in \rr{m}$ such that $\mathrm{F}_{ij} \in \on{Ob}\csym{D}(\mathtt{i,i})$. We have shown that, in that case, there is $k \in \mathtt{N}_{L}$ such that $(e_{j}A)^{\ast} \simeq Ae_{k}$. Since $\mathtt{N}_{R} = \mathtt{N}_{L}$, this shows that $\mathtt{N}_{R}$ is a self-injective core.
 \end{proof}
 
 \begin{corollary}\label{DiagonalWeaklyFiat}
  A $2$-subcategory of $\csym{C}_{\!A}$ of the form $\csym{D}_{U_{1}\times U_{2}}$ is weakly fiat if and only if $U_{1} = U_{2}$ and $U_{1}$ is a self-injective core.
 \end{corollary}

 \begin{proof}
  Assume that $\csym{D}$ is weakly fiat. From the definition we have $U_{1} = \mathtt{N}_{L}$ and $U_{2} = \mathtt{N}_{R}$. From Lemma \ref{LRLemma} it follows that $U_{1} = U_{2}$ and that $U_{1}$ is a self-injective core.
  
  Assume that $U_{1}$ is a self-injective core and let $\mathrm{F}_{ij}$ be an indecomposable $1$-morphism of $\csym{D}_{U\times U}$. Let $k \in U_{1}$ be such that $(e_{j}A)^{\ast} \simeq Ae_{k}$. Then, from Lemma \ref{AdjSelf} it follows that $\mathrm{F}_{ki}$ is right adjoint to $\mathrm{F}_{ij}$. Hence, $\csym{D}_{U_{1}\times U_{1}}$ is weakly fiat.
 \end{proof}

 \begin{proposition}
  Let $U$ be a self-injective core for $A$ and let $e = \sum_{i \in U} e_{i}$. The centralizer subalgebra $eAe$ is self-injective.
 \end{proposition}

 \begin{proof}
  The set $\setj{e_{i} \; | \; i \in U}$ is a \idem \ for $eAe$. Given $i \in U$, the functor $\on{Hom}_{A\!\on{-mod}}(Ae,-): A\!\on{-mod} \rightarrow eAe\!\on{-mod}$ sends the indecomposable projective $Ae_{i}$ to the indecomposable projective $eAe_{i}$ and the indecomposable injective $(e_{i}A)^{\ast}$ to the indecomposable injective $(e_{i}Ae)^{\ast}$. Since $U$ is a self-injective core, we may choose $j \in U$ such that $Ae_{i} \simeq (e_{j}A)^{\ast}$. Then
  \[
   (e_{j}Ae)^{\ast} \simeq \on{Hom}_{A\!\on{-mod}}(Ae,(e_{j}A)^{\ast}) \simeq \on{Hom}_{A\!\on{-mod}}(Ae,Ae_{i}) \simeq eAe_{i},
  \]
  which shows that $eAe$ indeed is self-injective.
 \end{proof}
 
   It is easy to verify that, given two subsets $U_{1},U_{2}$ of $\rr{m}$, the set $\left(U_{1}\times U_{2}\right) \cup \setj{0,1}$ gives a submonoid of $\mathcal{N}(A)$.
  
 \begin{definition}
  Let $U,V$ be subsets of $\rr{m}$. We say that a combinatorial $2$-subcategory $\csym{D}$ of $\csym{C}_{A}$ is {\it $U$-superdiagonal} if $U \subseteq V$ and $\csym{D} = \csym{D}_{U \times V}$. We say that $\csym{D}$ is {\it $U$-subdiagonal} if $V \subseteq U$ and $\csym{D} = \csym{D}_{U \times V}$.
  If $\csym{D}$ is both $U$-superdiagonal and $U$-subdiagonal, then $\csym{D} = \csym{D}_{U\times U}$ and we say that $\csym{D}$ is {\it $U$-diagonal.}
 \end{definition}
 
 In this document, we will focus on the $U$-superdiagonal case. The remaining results of this section also hold in the $U$-subdiagonal case, although Proposition \ref{SidedPreservation} must be modified as described in its formulation. Further, in the $U$-superdiagonal case, the vacuous $J$-cells also constitute left cells, whereas in the $U$-subdiagonal case these constitute right cells. The crucial difference, which is the reason for restricting our attention to the $U$-superdiagonal case, is that Proposition \ref{TransitiveRestriction} is not true in the $U$-subdiagonal case. Neither is the main result of this document, Theorem \ref{MainResult}: a counterexample is given by \cite[Theorem~5.10]{St}.
 
 \begin{proposition}
  A combinatorial $2$-subcategory $\csym{D}$ of $\csym{C}_{\!A}$ is $U$-superdiagonal if and only if it contains $\csym{D}_{U\times U}$ and $\mathtt{N}_{L}(\csym{D}) = U$.
 \end{proposition}

 \begin{proof}
  Clearly, $\csym{D}$ being $U$-superdiagonal implies both that $\csym{D}$ contains $\csym{D}_{U\times U}$ and that $\mathtt{N}_{L}(\csym{D}) = U$. Assume that $\csym{D}$ contains $\csym{D}_{U \times U}$ and $\mathtt{N}_{L} = U$. Let $j \in \mathtt{N}_{R}$ and let $i \in U$ be such that $\mathrm{F}_{ij} \in \on{Ob}\csym{D}(\mathtt{i,i})$. 
  Given $k \in U$, from the assumption we know that $\mathrm{F}_{ki}$ is a $1$-morphism of $\csym{D}$. 
  Since $e_{i}Ae_{i} \neq 0$, the $1$-morphism $\mathrm{F}_{kj}$ is a direct summand of $\mathrm{F}_{ki} \circ \mathrm{F}_{ij}$.
  Since $\csym{D}$ is finitary, it follows that $\mathrm{F}_{kj}$ is a $1$-morphism of $\csym{D}$. This shows that $\csym{D} = \csym{D}_{U \times \mathtt{N}_{R}(\ccf{D})}$.
 \end{proof}

  Given a graph $\Gamma$, we use the notation of \cite{ET} and denote the {\it zigzag algebra} on $\Gamma$ by $\zigzag(\Gamma)$. As defined in \cite{HK}, the algebra $\zigzag(\Gamma)$ is the quotient of the path algebra of the double quiver on $\Gamma$ by the ideal generated by paths $i \rightarrow j \rightarrow k$ for $i \neq k$, together with elements of the form $\alpha - \beta$, where $\alpha, \beta$ are different $2$-cycles at the same vertex of $\Gamma$. 
  From \cite[Proposition~1]{HK}, we know that a zigzag algebra $\zigzag(\Gamma)$ is weakly symmetric. Let $S_{k}$ be the star graph on $k+1$ vertices, labelled as follows:
  
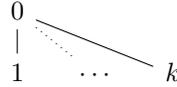
\[
   \begin{tikzcd}[sep = small]
    0 \arrow[d, no head] \arrow[dr, dotted, no head] \arrow[drr, no head] \\
    1 & \cdots & k
   \end{tikzcd}
  \]
  Following \cite{Zi}, we call a zigzag algebra of the form $\zigzag(S_{k})$ a {\it star algebra.} The above labelling of the vertices of $S_{k}$ induces a \idem\ $\setj{e_{0},e_{1},\ldots, e_{k}}$ for $\starv{k}$. 
 \begin{example}\label{StarAlgebras}
  Consider the star algebra $A := \starv{2}$, by definition given as the quotient of the path algebra of
  
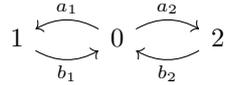
\[
   \begin{tikzcd}
   1 \arrow[r, bend right, swap, "b_{1}"] & 0 \arrow[r, bend left, "a_{2}"] \arrow[l, swap, bend right, "a_{1}"] & 2 \arrow[l, bend left, "b_{2}"]
   \end{tikzcd}
  \]
 by the ideal generated by $\setj{a_{2}b_{1}, a_{1}b_{2}, b_{2}a_{2}-b_{1}a_{1}}$. 
  The set $\setj{1,0}$ is a self-injective core for $A$, and the shaded part of the eggbox diagram of $\csym{C}_{\!A}$ below corresponds to the $\setj{0,1}$-superdiagonal $2$-subcategory $\csym{D}_{\setj{1,0}\times \setj{1,0,2}}$.
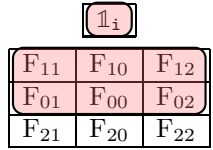
\begin{gather*}
 \begin{array}{|c|}
  \hline \tikzmarkup{left}
  \mathbb{1}_{\mathtt{i}}
  \tikzmarkdown{right}\\
  \hline
 \end{array}
\DrawRedBox[thick] \\
\begin{array}{|c|c|c|}
\hline
\tikzmarkup{left2}\mathrm{F}_{11}&\mathrm{F}_{10}&\mathrm{F}_{12}\\
\hline
\mathrm{F}_{01}&\mathrm{F}_{00}&\mathrm{F}_{02}\tikzmarkdown{right2}\\
\hline
\mathrm{F}_{21}&\mathrm{F}_{20}&\mathrm{F}_{22}\\
\hline
\end{array}\\
\DrawRedBoxZ[thick]
\end{gather*}
 \end{example}
 
\begin{example}
 A weakly fiat, combinatorial $2$-subcategory of $\csym{C}_{\!A}$ is not necessarily of the form $\csym{D}_{U\times U}$. Consider again the star algebra $\starv{2}$ of Example \ref{StarAlgebras}. We have
 \[
  \mathrm{F}_{11}\circ \mathrm{F}_{11} \simeq \mathrm{F}_{11}, \quad \mathrm{F}_{22}\circ \mathrm{F}_{22} \simeq \mathrm{F}_{22} \text{ and } \mathrm{F}_{11}\circ \mathrm{F}_{22} = \mathrm{F}_{22}\circ \mathrm{F}_{11} = 0. \]
 Hence, we may consider the combinatorial $2$-subcategory $\csym{D}_{\setj{(1,1),(2,2)}}$. Being a zigzag algebra, $\starv{2}$ is weakly symmetric, and so from Lemma \ref{AdjSelf} it follows that both $\mathrm{F}_{11}$ and $\mathrm{F}_{22}$ are self-adjoint. Hence $\csym{D}_{\setj{(1,1),(2,2)}}$ is fiat, and in particular also weakly fiat. But the set $\setj{(1,1),(2,2)}$ is not a product of subsets of $\setj{1,0,2}$.
\end{example}
 
 \begin{remark}
  A self-injective core is not unique for an algebra $A$ with a fixed \idem\ $\setj{e_{1},\ldots, e_{m}}$. If $A$ is weakly symmetric, then any non-empty subset of $\rr{m}$ gives a self-injective core, yielding $2^{m}-1$ different cores for $A$. More generally, if $A$ is self-injective with Nakayama permutation $\nu$ of $\rr{m}$, then the self-injective cores of $A$ are given by the unions of orbits of $\nu$.
  
  Moreover, not every algebra $A$ admits a self-injective core: if there are no non-zero projective-injective modules over $A$, then $A$ cannot have a self-injective core. The existence of such a module does not guarantee the existence of a self-injective core, either. A family of counterexamples is given by hereditary algebras of type $A$. Let us label the uniformly oriented quiver for $A_{n}$ in the standard way:
  \[
   \begin{tikzcd}
    1 \arrow[r] & 2 \arrow[r] & \cdots \arrow[r] & n
   \end{tikzcd}
  \]
 Then the unique non-zero projective-injective module is given by $A_{n}e_{1} \simeq (e_{n}A_{n})^{\ast}$. A self-injective core containing $1$ would thus have to contain $n$. However, $A_{n}e_{n}$ is not injective.
 
 In view of the above described non-uniqueness, our choice of terminology may seem peculiar. It is motivated by the essential role of the self-injective core $U$ and its associated $U$-diagonal $2$-subcategory of $\csym{C}_{\!A}$ in the classification of simple transitive $2$-representations of any $U$-superdiagonal $2$-subcategory of $\csym{C}_{\!A}$. 
 \end{remark}

 \begin{example}
  Let $\setj{e_{1},\ldots, e_{m}}$ be a \idem\ for $A$. Choose $U \subseteq \rr{m}$ and let $e = \sum_{i \in U}e_{i}$. Assume that $eAe$ is self-injective. This is not sufficient to conclude that $U$ is a self-injective core for $A$. Consider the algebra $A = \Bbbk Q/I$, where $Q$ is the quiver
  \[
   \begin{tikzcd}
    1 \arrow[r, bend left, "\alpha"] & 2 \arrow[l, bend left, "\beta"]
   \end{tikzcd}
  \]
 and $I$ is the ideal $\left\langle \alpha \beta \right\rangle$. The algebra $e_{2}Ae_{2}$ is isomorphic to $\Bbbk$, and hence in particular it is self-injective. But $\setj{2}$ is not a self-injective core, since the module $Ae_{2}$ is not injective.
 \end{example}

\subsection{\texorpdfstring{Cells of $U$-superdiagonal $2$-subcategories of $\csym{C}_{\!A}$}{Cell structure of U-superdiagonal 2-subcategories of C(A)}}\label{SuperDiagCellsSection}

For the remainder of this document, we let $U \subseteq \rr{m}$ be a self-injective core for $A$ and we let $\csym{D}$ be a $U$-superdiagonal $2$-subcategory of $\csym{C}_{\!A}$.

It is rather clear that the left, right and two-sided preorders of a combinatorial $2$-subcategory $\csym{D}$ of $\csym{C}_{\!A}$ are coarser than the restrictions of the respective preorders for $\csym{C}_{\!A}$ to the set $\mathcal{S}(\csym{D})$.
In many cases they are not strictly coarser:
\begin{proposition}\label{SidedPreservation}
\hspace{2em}
\begin{enumerate}[label = (\alph*)]
 \item  The left preorder of $\csym{D}$ coincides with the restriction of the left preorder of $\csym{C}_{\!A}$ to $\mathcal{S}(\csym{D})$.
 \item
 Dually, if $\csym{B}$ is a $U$-subdiagonal $2$-subcategory of $\csym{C}_{\!A}$, then the right preorder of $\csym{B}$ concides with the restriction of the right preorder of $\csym{C}_{\!A}$ to $\mathcal{S}(\csym{B})$.
 \item Hence, the left, right and two-sided preorders of $\csym{D}_{U\times U}$ all coincide with the respective restrictions.
\end{enumerate}
\end{proposition}
\begin{proof}
 We prove the first statement. The second one is dual, and the third one is an immediate consequence of the first two.

 Let $\leq_{L}^{A}$ denote the left preorder of $\csym{C}_{\!A}$ and let $\leq_{L}^{\ccf{D}}$ denote the left preorder of $\csym{D}$. In view of the observation preceding the proposition, it suffices to show that, for indecomposable $1$-morphisms $\mathrm{F},\mathrm{G}$ of $\csym{D}$, the condition $\mathrm{F} \leq_{L}^{A} \mathrm{G}$ implies $\mathrm{F} \leq_{L}^{\ccf{D}} \mathrm{G}$. It is clear that $[\mathbb{1}_{\mathtt{i}}]$ remains the unique minimal element the cell structure of a combinatorial $2$-subcategory of $\csym{C}_{\!A}$, so we may write $\mathrm{F} = \mathrm{F}_{ij}$ and $\mathrm{G} = \mathrm{F}_{kl}$ with $i,k \in U$.
 $\mathrm{F}_{ij} \leq_{L}^{A} \mathrm{F}_{kl}$ is equivalent to $j = l$. We thus need to show that, for $i,k \in U$ and $j$ such that $\mathrm{F}_{ij}, \mathrm{F}_{kj} \in \on{Ob}\csym{D}(\mathtt{i,i})$, there is a $1$-morphism $H$ of $\csym{D}$ such that $\mathrm{F}_{kj}$ is a direct summand of $H \circ \mathrm{F}_{ij}$. Let $\mathrm{H} = \mathrm{F}_{ki}$. Since $i,k \in U$ and $\csym{D}$ contains $\csym{D}_{U \times U}$, we know that $\mathrm{F}_{ki}$ is a $1$-morphism of $\csym{D}$. And since $e_{i}Ae_{i} \neq 0$, we see that $\mathrm{F}_{kj}$ indeed is a direct summand of $\mathrm{F}_{ki} \circ \mathrm{F}_{ij} \simeq \mathrm{F}_{kj}^{\oplus \dim e_{i}Ae_{i}}$. The result follows.
\end{proof}

\begin{proposition}\label{VacuousCellsExist}
 Let $\mathrm{F} \in \mathcal{S}(\csym{D})\setminus\setj{[\mathbb{1}_{\mathtt{i}}]}$. Let let $i \in U$ and $j \in \rr{m}$ be such that $\mathrm{F} = \mathrm{F}_{ij}$.
 \begin{enumerate}[label = (\alph*)]
  \item The following are equivalent:
  \begin{enumerate}[label = (\roman*)]
  \item There is a $1$-morphism $\mathrm{G} \in \mathcal{S}(\csym{D}_{U\times U})$ such that $\mathrm{F}$ is $J$-equivalent to $G$ inside $\csym{D}$. \label{GProps}
  \item There is $h \in U$ such that $e_{j}Ae_{h} \neq 0$.
  \end{enumerate}
  \item If there is no $h$ as above, then $\mathrm{F}$ lies in a maximal $J$-cell of $\csym{D}$, which is not idempotent.
 \end{enumerate}
\end{proposition}

\begin{proof}
 Since $\mathrm{G}$ is not isomorphic to $\mathbb{1}_{\mathtt{i}}$, without loss of generality we may assume $\mathrm{G} = \mathrm{F}_{kl}$, for some $k,l \in U$. 
 
 Suppose that there is $h \in U$ such that $e_{j}Ae_{h} \neq 0$. We then have
 \[
  \mathrm{F}_{kl} \in \on{Ob}\on{add}\setj{\mathrm{F}_{ki}\circ \mathrm{F}_{ij} \circ \mathrm{F}_{hl}},
 \]
 where $h,i,k,l \in U$, so that $\mathrm{F}_{ki},\mathrm{F}_{hl} \in \on{Ob}\csym{D}(\mathtt{i,i})$. This shows that $\mathrm{F}_{kl} \geq_{J} \mathrm{F}_{ij}$. Further, we have $\mathrm{F}_{ij} \geq_{J} \mathrm{F}_{kl}$, since
 \[
  \mathrm{F}_{ij} \in \on{Ob}\on{add}\setj{\mathrm{F}_{ik}\circ \mathrm{F}_{kl} \circ \mathrm{F}_{lj}}.
 \]
 Note that we did not use the assumption about $h$ for this latter condition.
 
 Now assume that $G$ is as specified in \ref{GProps}.
 Again let $\mathrm{G} = \mathrm{F}_{kl}$, with $k,l \in U$. By definition, there are $\mathrm{H},\mathrm{H}'$ such that $\mathrm{F}_{kl} \simeq \mathrm{H} \circ \mathrm{F}_{ij} \circ \mathrm{H}'$.
 Due to the biadditivity of composition of $1$-morphisms, we may assume $\mathrm{H},\mathrm{H}'$ to be indecomposable. If we write $\mathrm{H}' = \mathrm{F}_{xy}$, then, clearly, it is necessary that $e_{j}Ae_{x} \neq 0$. Further, we have $x \in \mathtt{N}_{L}(\csym{D}) = U$. We may now let $h:=x$.
 
 If $\mathrm{F}_{ij}$ is not $J$-equivalent to the $1$-morphisms of $\csym{D}_{U\times U}$, then $e_{j}Ae_{h} = 0$, for all $h \in U$. But $\csym{D}$ is $U$-superdiagonal, hence in particular $\mathtt{N}_{L} = U$. This implies that, for any $\mathrm{H} \in \mathcal{S}(\csym{D})\setminus\setj{\mathbb{1}_{\mathtt{i}}}$,
 we have $\mathrm{F}_{ij} \circ \mathrm{H} = 0$. Thus, for $\mathrm{H}' \in \mathcal{S}(\csym{D})$, the statements $\mathrm{H}' \geq_{L}^{\ccf{D}} \mathrm{F}_{ij}$ and $\mathrm{H}' \geq_{J}^{\ccf{D}} \mathrm{F}_{ij}$ are equivalent. From Proposition \ref{SidedPreservation} we infer that the left cell $\mathcal{L}_{j} = \setj{[\mathrm{F}_{yj}] \; | \; y \text{ such that } \mathrm{F}_{yj} \in \on{Ob}\csym{D}(\mathtt{i,i})}$ is a maximal left cell in $\csym{D}$, and so 
 \[
\mathcal{L}_{j} = \setj{\mathrm{H}' \in \mathcal{S}(\csym{D}) \; | \; \mathrm{H}' \geq_{J}^{\ccf{D}} \mathrm{F}_{ij}}.
 \]
 It follows that $\mathcal{L}_{j}$ is a maximal $J$-cell of $\csym{D}$. All of its elements are annihilated by right composition with non-identity indecomposable $1$-morphisms, and so composition of any two elements of $\mathcal{L}_{j}$ is zero, which proves that $\mathcal{L}_{j}$ is not idempotent.
\end{proof}

As a consequence, $\csym{D}$ has exactly two idempotent $J$-cells. The $J$-minimal among the two is given by $\mathbb{1}_{\mathtt{i}}$, the other is the $J$-cell containing  $F_{ij} \in \mathcal{S}(U\times U)$. We denote the former by $\mathcal{J}_{0}^{\ccf{D}}$ and the latter by $\mathcal{J}_{1}^{\ccf{D}}$. If there is no risk of ambiguity, we may omit the superscript $\csym{D}$ and write $\mathcal{J}_{0},\mathcal{J}_{1}$.

The left and right cell structures of $\csym{D}$ restricted to the union of the idempotent $J$-cells $\mathcal{J}_{0}^{\ccf{D}},\mathcal{J}_{1}^{\ccf{D}}$ is given by the restriction of the respective cell structures of $\csym{C}_{\!A}$. 
Additionally, $\csym{D}$ admits a (possibly empty) set of mutually incomparable, $J$-maximal, non-idempotent $J$-cells, each strictly $J$-greater than the idempotent $J$-cells. Such a $J$-cell is also a left cell, and the right cells inside it are singletons. We refer to such $J$-cells as {\it vacuous cells}.

The terminology is motivated by the fact that vacuous cells can be ignored when considering our problem of classification of simple transitive $2$-representations. 
A vacuous cell $\mathcal{J}$ is $J$-maximal and non-idempotent, and hence, as a consequence of \cite[Proposition~3]{CM}, it is annihilated by every simple transitive $2$-representation, so we may replace $\csym{D}$ by its quotient by the $2$-ideal generated by $\setj{ \on{id}_{F} \; | \; F \in \mathcal{J}}$. 

\begin{proposition}\label{Cell2Reps}
 Let $\mathcal{L},\mathcal{L}'$ be two left cells of $\csym{D}$. The cell $2$-representations $\mathbf{C}_{\mathcal{L}}, \mathbf{C}_{\mathcal{L}'}$ are equivalent if and only if they have the same apex.
\end{proposition}

\begin{proof}
 
 Equivalent $2$-representations have the same apex, so we only need to prove that cell $2$-representations with the same apex are equivalent.
  From Propositions \ref{SidedPreservation}, \ref{VacuousCellsExist} we conclude that the apex of a cell $2$-represenation not associated to the minimal left cell $\setj{[\mathbb{1}_{\mathtt{i}}]}$ is $\mathcal{J}_{1}$. 
  Indeed, if $F_{kl} \in \mathcal{L}$, then $\mathrm{F}_{kk} \circ \mathrm{F}_{kl} \neq 0$ shows that the apex of $\mathbf{C}_{\mathcal{L}}$ is $\mathcal{J}_{1}$. This shows that $\mathcal{J}_{0}$ is the unique left cell whose cell $2$-representation has $\mathcal{J}_{0}$ as apex. It thus suffices to prove the claim for $\mathcal{L},\mathcal{L}' \neq \mathcal{J}_{0}$. 
  In that case, there are $j,j' \in \mathtt{N}_{R}(\csym{D})$ such that $\mathcal{L} = \mathcal{S}(U \times \setj{j})$ and $\mathcal{L}' = \mathcal{S}(U \times \setj{j}')$.
  
  Recall that $\mathbf{C}_{\mathcal{L}} = \mathbf{N}_{\mathcal{L}}/\mathbf{I}_{\mathcal{L}}$, where the target category of $\mathbf{N}_{\mathcal{L}}$ is $\on{add}\mathcal{S}(U \times \setj{j})$. 
  Consider the left cell $\overline{\mathcal{L}} = \mathcal{S}([m] \times \setj{j})$ of $\csym{C}_{A}$. 
  Similarly, we have $\mathbf{C}_{\overline{\mathcal{L}}} = \mathbf{N}_{\overline{\mathcal{L}}}/\mathbf{I}_{\overline{\mathcal{L}}}$, with $\mathbf{N}_{\overline{\mathcal{L}}}(\mathtt{i}) = \on{add}\mathcal{S}([m] \times \setj{j})$. 
  Let $e_{U} = \sum_{i\in U}e_{i}$. By the definition of $\mathrm{F}_{U \times \setj{j}}$, there is a canonical algebra isomorphism $\on{End}_{\ccf{D}(\mathtt{i,i})}(\mathrm{F}_{U\times \setj{j}}) \simeq e_{U}Ae_{U} \otimes_{\Bbbk} e_{j}Ae_{j}$. 
  One may easily verify that $\mathbf{I}_{\mathcal{L}}$ is determined on the level of indecomposable objects by the ideal $e_{U}A e_{U} \otimes_{\Bbbk} e_{j}(\on{Rad}A)e_{j}$. Similarly, $\on{End}_{\ccf{C}_{A}(\mathtt{i,i})}(\mathrm{F}_{[m]\times \setj{j}}) \simeq A \otimes_{\Bbbk} e_{j}Ae_{j}$ and $\mathbf{I}_{\overline{\mathcal{L}}}$ is determined by the ideal $A \otimes_{\Bbbk} e_{j}(\on{Rad}A)e_{j}$. This observation is used in \cite[Proposition~9]{MM5}.
  
  The inclusion of $\mathrm{F}_{U \times \setj{j}}$ in $\mathrm{F}_{[m] \times \setj{j}}$ corresponds to the canonical inclusion of $e_{U}Ae_{U} \otimes_{\Bbbk} e_{j}Ae_{j}$ in $A \otimes_{\Bbbk} e_{j}Ae_{j}$. It follows that $\mathbf{C}_{\mathcal{L}}$ is a $2$-subrepresentation of the restriction of $\mathbf{C}_{\overline{\mathcal{L}}}$ to a $2$-representation of $\csym{D}$. 
  
  From \cite[Proposition~9]{MM5}, we know that the functor $\mathbf{C}_{\overline{\mathcal{L}}}(\mathtt{i}) \rightarrow \mathbf{C}_{\overline{\mathcal{L}'}}(\mathtt{i})$, given by sending $F_{ij}$ to $F_{ij'}$ on the level of objects, and on the level of morphisms corresponding to the map
  \[
   (A \otimes_{\Bbbk} e_{j}Ae_{j})/(A \otimes_{\Bbbk} e_{j}(\on{Rad}A)e_{j}) \simeq A \xrightarrow{\on{id}_{A}} A \simeq (A \otimes_{\Bbbk} e_{j'}Ae_{j'})/(A \otimes_{\Bbbk} e_{j'}(\on{Rad}A)e_{j'}),
  \]
  gives an equivalence $\mathbf{C}_{\overline{\mathcal{L}}} \xrightarrow{\sim} \mathbf{C}_{\overline{\mathcal{L}'}}$.
  Using the commutativity of
  \[
   \begin{tikzcd}[sep = small]
    (A \otimes_{\Bbbk} e_{j}Ae_{j})/(A \otimes_{\Bbbk} e_{j}(\on{Rad}A)e_{j}) \arrow[r, "\sim"] & A \\
    (e_{U}Ae_{U} \otimes_{\Bbbk} e_{j}Ae_{j})/(e_{U}Ae_{U} \otimes_{\Bbbk} e_{j}(\on{Rad}A)e_{j}) \arrow[r, "\sim"] \arrow[u, hook] & e_{U}Ae_{U} \arrow[u, hook]
   \end{tikzcd},
  \]
  we conclude that the equivalence $\mathbf{C}_{\overline{\mathcal{L}}} \xrightarrow{\sim} \mathbf{C}_{\overline{\mathcal{L}'}}$ restricts to an equivalence $\mathbf{C}_{\mathcal{L}} \simeq \mathbf{C}_{\mathcal{L}'}$, from which the result follows.
\end{proof}

\subsection{\texorpdfstring{$2$-representations of fiat $U$-diagonal $2$-subcategories of $\csym{C}_{\!A}$}{Simple transitive 2-representations of weakly fiat U-diagonal 2-subcategories of C(A)}}

Recall that, by Proposition \ref{DiagonalWeaklyFiat}, the $2$-category $\csym{D}_{U \times U}$ is weakly fiat. Let $\csym{D}$ be a combinatorial $2$-subcategory of $\csym{C}_{\!A}$ containing $\csym{D}_{U\times U}$. Let $\mathbf{M}$ be a simple transitive $2$-representation of $\csym{D}$. 
The following is an immediate consequence of \cite[Theorem~3.1]{Zi}:
\begin{lemma}\label{DiagonalActionProjective}
  Given a $1$-morphism $\mathrm{F} \in \on{Ob}\csym{D}_{U\times U}(\mathtt{i,i})$, the functor $\overline{\mathbf{M}}\mathrm{F}$ is a projective functor.
\end{lemma}

Recall from \cite[Section~4.8]{MM1} that 
 a $J$-cell of a finitary $2$-category is called {\it strongly regular} if
 \begin{itemize}
  \item any two left (respectively right) cells in $\mathcal{J}$ are not comparable with respect to the left (respectively right) order;
  \item the intersection between any right and any left cell in $\mathcal{J}$ is a singleton.
 \end{itemize}
 
From the results of the preceding section, it is clear that all the $J$-cells of $\csym{D}$ are strongly regular.
 Combining that with Proposition \ref{DiagonalWeaklyFiat}, we find that $\csym{D}_{U\times U}$ is weakly fiat with strongly regular $J$-cells. The following is then an immediate consequence of \cite[Theorem~33]{MM6}:
\begin{proposition}\label{StronglyRegularUDiagonal}
 Any simple transitive $2$-representation of $\csym{D}_{U\times U}$ is equivalent to a cell $2$-representation.
\end{proposition}

For the remainder of this document, let $\mathcal{L}_{0}$ be the left cell of $\csym{D}_{U \times U}$ consisting of $[\mathbb{1}_{\mathtt{i}}]$ and let $\mathcal{L}_{1}$ be a left cell contained in $\mathcal{J}_{1}^{\ccf{D}_{U\times U}}$.
Since any left cell of $\csym{D}_{U \times U}$ is also a left cell of $\csym{D}$, both $\mathcal{L}_{0}$ and $\mathcal{L}_{1}$ give left cells of both $\csym{D}$ and $\csym{D}_{U \times U}$. We denote the associated $2$-representations of $\csym{D}$ by $\mathbf{C}_{\mathcal{L}_{0}}, \mathbf{C}_{\mathcal{L}_{1}}$, and the associated $2$-representations of $\csym{D}_{U\times U}$ by $\mathbf{C}_{\mathcal{L}_{0}}^{U}, \mathbf{C}_{\mathcal{L}_{1}}^{U}$.

\begin{proposition}\label{DExtUDiagonal}
We have:
 \begin{enumerate}
  \item $\on{Dext}(\mathbf{C}_{\mathcal{L}_{0}}^{U}, \mathbf{C}_{\mathcal{L}_{0}}^{U}) = \varnothing$
  \item $\on{Dext}(\mathbf{C}_{\mathcal{L}_{1}}^{U}, \mathbf{C}_{\mathcal{L}_{1}}^{U}) = \varnothing$
  \item $\on{Dext}(\mathbf{C}_{\mathcal{L}_{1}}^{U}, \mathbf{C}_{\mathcal{L}_{0}}^{U}) = \varnothing$.
 \end{enumerate}
\end{proposition}

\begin{proof}
 The above statement is a minor modification of \cite[Theorem~6.22]{CM}, and the proof therein requires only two changes for our case. 
 First, we  restrict the index set, from $\rr{n}$ in \cite{CM}, to $U$ in our case. 
 Second, rather than assume that $A$ is self-injective, we only assume that $U$ is a self-injective core. Both the assumptions imply the properties which are used in the proof given in \cite{CM}. 
\end{proof}

Proposition \ref{DExtUDiagonal} immediately implies the following statement:

\begin{corollary}\label{UDiagonalMatrices}
  Let $\mathbf{M}$ be a finitary $2$-representation of $\csym{D}_{U \times U}$. There is a labelling of the indecomposable objects of $\mathbf{M}(\mathtt{i})$, with respect to which we have
\begin{displaymath}
[\mathrm{F}_{U\times U}]_{\mathbf{M}} = 
\left(
\begin{array}{c|c}
 \left(
 \begin{array}{c|c|c}
  [\mathrm{F}_{U\times U}]_{\mathbf{C}_{\mathcal{L}_{1}}^{U}} & 0 & 0 \\
  \hline
  0 & \ddots & 0 \\
  \hline
  0 & 0 & [\mathrm{F}_{U\times U}]_{\mathbf{C}_{\mathcal{L}_{1}}^{U}}
 \end{array}
 \right)
 & * \\
 \hline
 0 & 0
\end{array}
\right)
\end{displaymath}
Each diagonal block of the form $[\mathrm{F}_{U\times U}]_{\mathbf{C}_{\mathcal{L}_{1}}^{U}}$ corresponds to a simple transitive subquotient equivalent to $\mathbf{C}_{\mathcal{L}_{1}}^{U}$ in the weak Jordan-H{\" o}lder series of $\mathbf{M}$, and the bottom right diagonal block $0$ corresponds to the subquotients equivalent to $\mathbf{C}_{\mathcal{L}_{0}}^{U}$.
\end{corollary}

\section{The main result and its generalizations}
\label{s4}

\subsection{The main result}
Let $\zigzag(S_{k})$ be the star algebra defined in Example \ref{StarAlgebras}. The set $\setj{0}$ defines a self-injective core for $\zigzag(S_{k})$. Consider the $\setj{0}$-superdiagonal $2$-subcategory of $\csym{C}_{\!\zigzag(S_{k})}$ given by $\csym{D}_{\setj{0} \times \rr{0,k}}$. The main goal of this section is to generalize the following result:
\begin{theorem}\cite[Theorem~6.2]{Zi}
  Any simple transitive $2$-representation of the $2$-category $\csym{D}_{\setj{0}\times \rr{0,k}}$ is equivalent to a cell $2$-representation.
\end{theorem}

Given a finitary $2$-subcategory $\csym{B}$ of a finitary $2$-category $\csym{C}$, denote by 
\[
\on{Res}_{\ccf{B}}^{\ccf{C}}(-): \csym{C}\!\on{-afmod} \rightarrow \csym{B}\!\on{-afmod}
\]
the restriction $2$-functor given by precomposition with the inclusion $2$-functor from $\csym{B}$ to $\csym{C}$.
Let $\mathbf{M}$ be a finitary $2$-representation of $\csym{C}$. Following \cite[Theorem~8]{MM5}, we may choose a transitive subquotient $\mathbf{N}$ of $\on{Res}_{\ccf{B}}^{\ccf{C}}(\mathbf{M})$. For any $\mathtt{i} \in \on{Ob}\csym{C}$, we may label a complete, irredundant list of representatives of isomorphism classes of indecomposable objects of $\mathbf{M}(\mathtt{i})$ as $X_{1},\ldots, X_{l}$, so that there is $k \leq l$ such that the list $X_{1},\ldots, X_{k}$ of representatives is complete and irredundant for $\mathbf{N}(\mathtt{i})$. Under this labelling, the Cartan matrix $\mathtt{C}^{\mathbf{N}(\mathtt{i})}$ is the diagonal block submatrix of $\mathtt{C}^{\mathbf{M}(\mathtt{i})}$ corresponding to the indices $1,\ldots, k$. 
By \cite[Lemma~4]{MM5}, the $2$-representation $\mathbf{N}$ admits a unique maximal $\csym{B}$-stable ideal $\mathbf{I}$. The quotient $\mathbf{N}/\mathbf{I}$ is simple transitive. The projection $2$-transformation $\mathbf{N} \rightarrow \mathbf{N}/\mathbf{I}$ does not map any of the indecomposable objects to $0$, and so the Cartan matrices are of the same size, and, for $i,j \in \rr{k}$, we have the entry-wise inequality
\begin{equation}\label{EntryWiseIneq}
 \mathtt{C}^{\mathbf{M}(\mathtt{i})}_{ij} = \mathtt{C}^{\mathbf{N}(\mathtt{i})}_{ij} \geq \mathtt{C}^{\mathbf{N}(\mathtt{i})/\mathbf{I}(\mathtt{i})}_{ij} = \mathtt{C}^{(\mathbf{N}/\mathbf{I})(\mathtt{i})}_{ij}.
\end{equation}

For the remainder of this section, we will let $\mathbf{M}$ be a transitive $2$-representation of the $U$-superdiagonal $2$-subcategory $\csym{D}$ of $\csym{C}_{A}$, fixed in Section \ref{s3}. Similarly to \cite[Theorem~6.2]{Zi}, if the apex of $\mathbf{M}$ is $\mathcal{J}_{0}$, our remaining claims, including the main result, follow immediately. We thus assume that the apex of $\mathbf{M}$ is $\mathcal{J}_{1}$.

\begin{proposition}\label{TransitiveRestriction}
 The $2$-representation $\on{Res}_{\ccf{D}_{U\times U}}^{\ccf{D}}(\mathbf{M})$ is transitive.
\end{proposition}

\begin{proof}
 From the description of the cell structure of $\csym{D}$ given in Section \ref{SuperDiagCellsSection}, we know that the set $\mathcal{S}(U\times U)$ of indecomposable $1$-morphisms of $\csym{D}_{U\times U}$ is a union of left cells of $\csym{D}$. Hence, for any $\mathrm{G} \in \mathcal{S}(\csym{D})$, we have $\mathrm{G} \circ \mathrm{F}_{U\times U} \in \on{add}\mathcal{S}(U \times U)$.
 
 This shows that the additive closure of the essential images of the functors in $\setj{\mathbf{M}\mathrm{F} \; | \; \mathrm{F} \in \mathcal{S}(U \times U)}$ is stable under the functorial action of $\csym{C}$ given by $\mathbf{M}$. Using action notation, we may write
 \[
  \on{add}\left( \csym{D} \cdot \left(\mathrm{F}_{U\times U}\cdot \mathbf{M}(\mathtt{i})\right) \right) = \on{add}\left( \left(\csym{D} \circ \mathrm{F}_{U\times U}\right) \cdot \mathbf{M}(\mathtt{i})\right) = \on{add}\left(\mathrm{F}_{U\times U}\cdot \mathbf{M}(\mathtt{i})\right).
 \]
 Since the apex of $\mathbf{M}$ is $\mathcal{J}_{1}$, we have $\on{add}\setj{\mathrm{F}_{U\times U}\cdot \mathbf{M}(\mathtt{i})} \neq 0$. Further, by assumption, $\mathbf{M}$ is transitive, which yields $\on{add}\setj{\mathrm{F}_{U\times U}\cdot \mathbf{M}(\mathtt{i})} = \mathbf{M}(\mathtt{i})$.
 
 This shows that all the rows of the action matrix $[\mathrm{F}_{U\times U}]_{\mathbf{M}}$ are non-zero, and, in view of Corollary \ref{UDiagonalMatrices}, we find that all simple transitive weak subquotients of $\on{Res}_{\ccf{D}_{U\times U}}^{\ccf{D}}(\mathbf{M})$ are equivalent to $\mathbf{C}_{\mathcal{L}_{1}}^{U}$. 
 Thus, if we let $k$ be the length of the weak Jordan-H{\" o}lder series of $\on{Res}_{\ccf{D}_{U\times U}}^{\ccf{D}}\mathbf{M}$,
 Corollary \ref{UDiagonalMatrices} implies that the matrix 
 $[\mathrm{F}_{U\times U}]_{\mathbf{M}} = [\mathrm{F}_{U\times U}]${\raisebox{-1.5pt}{$\scriptstyle{\on{Res}_{\scaleobj{0.72}{\ccf{D}_{U\times U}}}^{\scaleobj{0.72}{\ccf{D}}}(\mathbf{M})}$}} is the $k \times k$ block diagonal matrix
  \begin{displaymath}
 \left(
 \begin{array}{c|c|c}
  [F_{U\times U}]_{\mathbf{C}_{\mathcal{L}_{1}}} & 0 & 0 \\
  \hline
  0 & \ddots & 0 \\
  \hline
  0 & 0 & [F_{U\times U}]_{\mathbf{C}_{\mathcal{L}_{1}}}
 \end{array}
 \right).
\end{displaymath}
 Observe that all the entries of $[\mathrm{F}_{U\times U}]_{\mathbf{C}_{\mathcal{L}_{1}}}$ are positive integers. Due to the additivity of action matrices, we have $[\mathrm{F}_{U\times U}]_{\mathbf{M}} = \sum_{(i,j) \in U\times U} [\mathrm{F}_{ij}]_{\mathbf{M}}$. 
 From our earlier observations it follows that $\mathrm{F}_{\mathcal{S}(\ccf{D})} \circ \mathrm{F}_{U \times U}$ has a direct sum decomposition with summands in $\mathcal{S}(U\times U)$. As a consequence, there are non-negative integers $\setj{a_{ij} \; | \; (i,j) \in U\times U}$ such that
\[
 [\mathrm{F}_{\mathcal{S}(\ccf{D})}\circ \mathrm{F}_{U\times U}]_{\mathbf{M}} = \sum_{(i,j) \in U \times U} a_{ij} [\mathrm{F}_{ij}]_{\mathbf{M}}.
\]
 Since $\mathcal{S}(U \times U)$ is a subset of $\mathcal{S}(\csym{D})$, the integers $a_{ij}$ are positive, for all $i,j \in U$. Using what we know about $[\mathrm{F}_{U\times U}]_{\mathbf{M}}$, we can write
  \begin{displaymath}
 [\mathrm{F}_{\mathcal{S}(\ccf{D})}\circ \mathrm{F}_{U\times U}]_{\mathbf{M}} = \left(
 \begin{array}{c|c|c}
  \mathtt{C}_{1} & 0 & 0 \\
  \hline
  0 & \ddots & 0 \\
  \hline
  0 & 0 & \mathtt{C}_{k}
 \end{array}
 \right),
 \end{displaymath}
 where the entries of the block submatrices $\mathtt{C}_{1}, \ldots, \mathtt{C}_{k}$ are positive integers. From the biadditivity of composition of $1$-morphisms we have
 \[
  [\mathrm{F}_{\mathcal{S}(\ccf{D})}\circ \mathrm{F}_{U\times U}]_{\mathbf{M}} = [\mathrm{F}_{\mathcal{S}(\ccf{D})}]_{\mathbf{M}} [\mathrm{F}_{U\times U}]_{\mathbf{M}}.
 \]
  Since $\mathbf{M}$ is transitive, the entries of the left factor in this product are positive integers, and we have shown that the entries of the right factor are non-negative integers with positive diagonal entries. This shows that all the entries of the matrix $[\mathrm{F}_{\mathcal{S}(\ccf{D})}\circ \mathrm{F}_{U\times U}]_{\mathbf{M}}$ are positive. Thus $k = 1$, and so $\on{Res}_{\ccf{D}_{U\times U}}^{\ccf{D}}(\mathbf{M})$ is transitive.
\end{proof}

\begin{corollary}\label{RestrictionMatricesRank}
 The rank of $\mathbf{M}$ coincides with the rank of $\mathbf{C}_{\mathcal{L}_{1}}$, which by definition equals $|U|$.
 Further, the action matrices of $1$-morphisms of $\csym{D}_{U\times U}$ for $\mathbf{M}$ coincide with those for $\mathbf{C}_{\mathcal{L}_{1}}$.
\end{corollary}
\begin{proof}
 Combining Proposition \ref{TransitiveRestriction} with Proposition \ref{StronglyRegularUDiagonal}, we conclude that the unique simple transitive quotient of $\on{Res}_{\ccf{D}_{U\times U}}^{\ccf{D}}(\mathbf{M})$ is equivalent to said cell $2$-representation. The claim follows.
\end{proof}

Assume that $\mathbf{M}$ is simple transitive.
Let $r:= |U|$. Without loss of generality, assume that $U = \rr{r}$. Let $Q$ be the basic algebra introduced in Section~\ref{CellSection}, satisfying $\mathbf{M}(\mathtt{i}) \simeq Q\!\on{-proj}$, together with the \idem\ described therein.
Given $s,t \in \rr{r}$, choose an endofunctor of $\overline{\mathbf{M}}(\mathtt{i}) \simeq Q\!\on{-mod}$ naturally isomorphic to the indecomposable projective functor $Qf_{s} \otimes_{\Bbbk} f_{t}Q \otimes_{Q} -$, and denote it by $\mathrm{G}_{st}$. This is consistent with the notation $\setj{\mathrm{F}_{ij} \; | \; i,j \in \rr{m}}$ for indecomposable $1$-morphisms of $\csym{C}_{\!A}$, since such $1$-morphisms correspond to the indecomposable projective endofunctors of $A\!\on{-mod}$. 
Using this notation, Lemma \ref{DiagonalActionProjective} implies
$\overline{\mathbf{M}}\mathrm{F} \in \on{Ob}\on{add}\setj{\mathrm{G}_{st} \; | \; s,t \in \rr{r}}$. 

We recall a notation convention introduced in \cite{MZ2}. Given $i,j \in U$, we write
\[
\begin{aligned}
 \mathtt{X}_{ij} = 
 \setj{
 \begin{aligned}
 &s \in \rr{r} \; | \; \text{there is } t \in \rr{r} \text{ such that } \mathrm{G}_{st}
 \text{ is}\\
 &\text{isomorphic to a direct summand of } \overline{\mathbf{M}}\mathrm{F}_{ij}
 \end{aligned}
 }, \\
  \mathtt{Y}_{ij} = 
 \setj{
 \begin{aligned}
 &t \in \rr{r} \; | \; \text{there is } s \in \rr{r} \text{ such that } \mathrm{G}_{st}
 \text{ is}\\
 &\text{isomorphic to a direct summand of } \overline{\mathbf{M}}\mathrm{F}_{ij}
 \end{aligned}
 }.
\end{aligned}
\]
Clearly, the essential image of $\mathrm{G}_{st}$ in $\overline{\mathbf{M}}(\mathtt{i})$ is given by $\on{add}\setj{Qe_{s}}$. Hence, $\mathtt{X}_{ij}$ coincides with the set of non-zero rows of $[\mathrm{F}_{ij}]_{\mathbf{M}}$. From Corollary \ref{RestrictionMatricesRank} we know that only the $i$th row of $[\mathrm{F}_{ij}]_{\mathbf{M}}$ is non-zero. Hence $\mathtt{X}_{ij} = \setj{i}$ for all $i,j \in U$. 
The following statement, as well as its proof, is completely analogous to \cite[Lemma~20]{MZ2} and \cite[Lemma~22]{MZ2}:
\begin{lemma}\label{MZ2Mimic}
 For any $i,j \in U$ we have $\mathtt{Y}_{ij} = \setj{j}$. Thus, $\mathbf{M}F_{ij} \simeq G_{ij}^{\oplus m_{ij}}$ for some positive integers $m_{ij}$.
\end{lemma}

\begin{theorem}\label{MainResult}
 Let $\csym{D}$ be a $U$-superdiagonal $2$-subcategory of $\csym{C}_{A}$.  Any simple transitive $2$-representation of $\csym{D}$ is equivalent to a cell $2$-representation.
\end{theorem}

\begin{proof}
 In view of our prior observations regarding $2$-representations with apex $\mathcal{J}_{0}$, it suffices to show the claim for the above described simple transitive $2$-representation $\mathbf{M}$.
 Let $\mathtt{C}^{Q}$ denote the Cartan matrix of $Q\!\on{-proj}$ and let $\mathtt{C}^{eAe}$ denote the Cartan matrix of $eAe\!\on{-proj}$, where $e$ is the idempotent $e_{U} = \sum_{i \in U}e_{i}$ of $A$. 
 By definition we have $([\mathrm{F}_{ij}]_{\mathbf{C}_{\mathcal{L}_{1}}})_{il} = ([\mathrm{F}_{ij}]_{\mathbf{C}_{\mathcal{L}_{1}}^{U}})_{il} = \mathtt{C}^{eAe}_{jl}$, for $i,j,l \in U$.
 
 Corollary \ref{RestrictionMatricesRank} yields $[\mathrm{F}_{ij}]_{\mathbf{M}} = [\mathrm{F}_{ij}]_{\mathbf{C}_{\mathcal{L}_{1}}}$. By
 Lemma \ref{MZ2Mimic}, we have $([\mathrm{F}_{ij}]_{\mathbf{M}})_{il} = m_{ij}\mathtt{C}^{Q}_{jl}$. 
  To summarize, for $i,j,l \in U$, we have
 \begin{equation}\label{CartanMatrixEqualities}
  \mathtt{C}_{jl}^{eAe} = ([\mathrm{F}_{ij}]_{\mathbf{C}_{\mathcal{L}_{1}}})_{il} = ([\mathrm{F}_{ij}]_{\mathbf{M}})_{il} = m_{ij}\mathtt{C}^{Q}_{jl}.
 \end{equation}
 Proposition \ref{TransitiveRestriction} shows that the unique simple transitive quotient of $\on{Res}_{\ccf{D}_{U\times U}}^{\ccf{D}}(\mathbf{M})$ is equivalent to $\mathbf{C}_{\mathcal{L}_{1}}^{U}$. Hence the inequality \eqref{EntryWiseIneq} implies that
 \begin{equation}\label{IneqConsequence}
  \mathtt{C}^{Q}_{ij} \geq \mathtt{C}^{eAe}_{ij} \text{, for } i,j \in U.
 \end{equation}
 Combining \eqref{CartanMatrixEqualities} and \eqref{IneqConsequence} we see that $m_{ij} = 1$ and $\mathtt{C}^{Q}_{jl} = \mathtt{C}^{eAe}_{jl}$, for $i,j,l \in U$. The result now follows from Lemma \ref{StdArg}.
\end{proof}

\subsection{Generalization to \texorpdfstring{$\csym{C}_{A,X}$}{C_{A,X}}}\label{CAX}
Let $Z$ be the algebra of $2$-endomorphisms of $\mathbb{1}_{\mathtt{i}} \in \on{Ob}\csym{C}_{A}(\mathtt{i},\mathtt{i})$ which factor through any $1$-morphism in $\on{add}\mathcal{S}(\rr{m} \times \rr{m})$. Given any subalgebra $X$ of $\on{End}_{\ccf{C}_{\!A}(\mathtt{i,i})}(\mathbb{1}_{\mathtt{i}})$ containing $Z$, the $2$-category $\csym{C}_{A,X}$, initially introduced in \cite[Section~4.5]{MM3}, is defined by having the same collection of $1$-morphisms as $\csym{C}_{A}$ and the same spaces of $2$-morphisms as $\csym{C}_{A}$ except $\on{End}_{\ccf{C}_{A,X}(\mathtt{i,i})}(\mathbb{1}_{\mathtt{i}}) = X$.

Importantly, $\csym{C}_{A,X}$ contains all adjunction $2$-morphisms of $\csym{C}_{A}$, and thus the inclusion $2$-functor $\csym{C}_{A,X} \hookrightarrow \csym{C}_{A}$ yields a bijection between adjunctions in the former and the latter. Observe that beyond our use of adjunctions, our arguments do not involve $\on{End}_{\ccf{C}_{\!A}(\mathtt{i,i})}(\mathbb{1}_{\mathtt{i}})$, hence our results generalize verbatim from $\csym{C}_{A}$ to $\csym{C}_{A,X}$, for any $X$.

\section{Self-injective cores and Duflo involutions}\label{s5}

One of the central tools used in the study of finitary $2$-repre\-sentations of weakly fiat $2$-cate\-gories is the use of {\it Duflo involutions}. Importantly, these were used for the first definition of a cell $2$-representation in \cite[Section~4.5]{MM1}. Other important applications and connections with other $2$-representation theoretic concepts can be found in \cite[Section~4]{MMMTZ2} and \cite[Section~6.3]{MMMT}.

Let $\csym{C}$ be a weakly fiat $2$-cate\-gory. Let $\mathtt{i} \in \on{Ob}\csym{C}$, and let $\mathcal{L}$ be a left cell of $\csym{C}$ whose elements have $\mathtt{i}$ as their domain. Consider the abelianized principal $2$-repre\-sentation $\overline{\mathbf{P}}_{\mathtt{i}} = \overline{\csym{C}(\mathtt{i},-)}$. Let $\hat{P}_{\mathbb{1}_{\mathtt{i}}}$ be the image of $\mathbb{1}_{\mathtt{i}}$ under the canonical embedding $\mathbf{P}_{\mathtt{i}}(\mathtt{i}) \hookrightarrow \overline{\mathbf{P}}_{\mathtt{i}}(\mathtt{i})$. By \cite[Proposition~27]{MM6}, there is a unique submodule $K$ of $\hat{P}_{\mathbb{1}_{\mathtt{i}}}$ such that every simple subquotient $\hat{P}_{\mathbb{1}_{\mathtt{i}}}/K$ is annihilated by any $\mathrm{F} \in \mathcal{L}$, and such that $K$ has a simple top $L$, with $\overline{\mathbf{P}}_{\mathtt{i}}\mathrm{F}(L) \neq 0$, for any $\mathrm{F} \in \mathcal{L}$. The projective cover $\mathrm{G}_{\mathcal{L}}$ of $K$ lies in $\overline{\mathbf{P}}_{\mathtt{i}}(\mathtt{i})\!\on{-proj} \simeq \mathbf{P}_{\mathtt{i}}(\mathtt{i}) \simeq \csym{C}(\mathtt{i},\mathtt{i})$. Abusing notation, we denote by $\mathrm{G}_{\mathcal{L}}$ the $1$-morphism of $\csym{C}$ corresponding to $\mathrm{G}_{\mathcal{L}}$ under this equivalence. The {\it Duflo involution in $\mathcal{L}$} is the $1$-morphism $\mathrm{G}_{\mathcal{L}}$.

Let $\csym{C}$ be a weakly fiat $2$-category with a unique object and let $\mathcal{J}$ be a strongly regular idempotent $J$-cell of $\csym{C}$. Using \cite[Theorem~4.28]{MMMTZ2}, rather than study the simple transitive $2$-representations of $\csym{C}$ with apex $\mathcal{J}$, one may equivalently study the simple transitive $2$-representations of an associated $2$-category $\csym{C}_{\mathcal{J}}$ with apex $\mathcal{J}$. Indeed, the respective $2$-categories of simple transitive $2$-representations are biequivalent. Further, this biequivalence sends cell $2$-representations to cell $2$-representations.

By definition, $\csym{C}_{\mathcal{J}}$ has only two $J$-cells - one containing $\mathbb{1}_{\mathtt{i}}$ and one given by $\mathcal{J}$. Further, it is $\mathcal{J}$-simple (see \cite[Section~6.2]{MM2}). It follows from \cite[Theorem~32]{MM6} that $\csym{C}_{\mathcal{J}}$ is biequivalent to $\csym{C}_{A,X}$, for a weakly symmetric $A$ and a suitable algebra $X \subseteq \on{End}_{\ccf{C}_{A}(\mathtt{i,i})}(\mathbb{1}_{\mathtt{i}})$. For our purposes we may thus restrict ourselves to the case $\csym{C} := \csym{C}_{A,X}$.

Let $\mathcal{L}$ be a left cell of $\mathcal{J}$. Let $\mathrm{G}_{\mathcal{L}}$ be the Duflo involution in $\mathcal{L}$, and let $L_{\mathrm{G}_{\mathcal{L}}}$ be the simple top of the image of $\mathrm{G}$ in $\overline{\mathbf{P}}_{\mathtt{i}}$, which is used to define the Duflo involution. Following \cite[Proposition~22]{MM2}, the finitary $2$-subrepresentation of $\overline{\mathbf{P}}_{\mathtt{i}}$ given by $\on{add}\setj{\overline{\mathbf{P}}_{\mathtt{i}}(\mathrm{F}\mathrm{G}_{\mathcal{L}})L_{\mathrm{G}_{\mathcal{L}}} \; | \; \mathrm{F} \in \csym{C}(\mathtt{i,i})}$ is equivalent to $\mathbf{C}_{\mathcal{L}}$. We conclude that $\overline{\mathbf{P}}_{\mathtt{i}}(\mathrm{G}_{\mathcal{L}})L_{\mathrm{G}_{\mathcal{L}}}$ is a generator for $\mathbf{C}_{\mathcal{L}}$ in the sense of \cite[Definition~2.19]{MMMTZ1}.

Given a $1$-morphism $\mathrm{F}$ of $\csym{C}$, let ${}^{\ast}\mathrm{F}$ denote its left adjoint. By \cite[Proposition~28]{MM6}, the right cell ${}^{\ast}\!\mathcal{L} = \setj{{}^{\ast}\mathrm{F} \; | \; \mathrm{F} \in \mathcal{L}}$ contains $\mathrm{G}_{\mathcal{L}}$. For ${}^{\ast}\mathrm{F} \in {}^{\ast}\!\mathcal{L}$, we have $({}^{\ast}\mathrm{F})\mathrm{G}_{\mathcal{L}} \in \on{add}\setj{\mathrm{G}_{\mathcal{L}}}$. Further, if $\mathrm{F} \in \mathcal{S}(\csym{C})\setminus\left(\setj{\mathbb{1}_{\mathtt{i}}}\cup {}^{\ast}\!\mathcal{L}\right)$, then an immediate consequence of \cite[Proposition~17{$(b)$}]{MM1} is that
$\overline{\mathbf{P}}_{\mathtt{i}}\left(\mathrm{F}\mathrm{G}_{\mathcal{L}}\right)L_{\mathrm{G}_{\mathcal{L}}}$ has no summand in $\on{add}\setj{\mathrm{G}_{\mathcal{L}}L_{\mathrm{G}_{\mathcal{L}}}}$. 

As a consequence, the superdiagonal $2$-subcategory $\csym{D}_{{}^{\ast}\!\mathcal{L}}$ of $\csym{C}$ given by the right cell ${}^{\ast}\!\mathcal{L}$ is precisely the $2$-subcategory of $\csym{C}$ {\it stabilizing} the subcategory $\on{add}\setj{\mathrm{G}_{\mathcal{L}}L_{\mathrm{G}_{\mathcal{L}}}}$, the additive closure of the generator of $\mathbf{C}_{\mathcal{L}}$. If $\csym{C}$ is fiat, then $\csym{D}_{{}^{\ast}\!\mathcal{L}}$ is additionally a self-injective core, and so, using the generalization of Theorem~\ref{MainResult} from Section~\ref{CAX}, we conclude that $\on{add}\setj{\mathrm{G}_{\mathcal{L}}L_{\mathrm{G}_{\mathcal{L}}}}$ is uniquely characterized as the target for the unique (up to equivalence) simple transitive $2$-representation of $\csym{D}_{{}^{\ast}\!\mathcal{L}}$ which does not annihilate the Duflo involution $\mathrm{G}_{\mathcal{L}}$.

More generally, given a self-injective core $U \subseteq \rr{m}$ such that $\mathrm{G}_{\mathcal{L}}$ lies in $\csym{D}_{U \times U}$, the $2$-subcategory $\csym{D}_{U \times \rr{m}}$ is precisely the $2$-subcategory stabilizing 
\[
\on{add}\setj{\overline{\mathbf{P}}_{\mathtt{i}}(\mathrm{F}\mathrm{G}_{\mathcal{L}})L_{\mathrm{G}_{\mathcal{L}}} \; | \; \mathrm{F} \in \mathcal{R}_{j} \text{, for some } j \in U},
\]
and, again as a consequence of Theorem~\ref{MainResult}, this category is characterized as the target for the unique simple transitive $2$-representation of $\csym{D}_{U \times \rr{m}}$ not annihilating $\mathrm{G}_{\mathcal{L}}$.

\noindent
Department of Mathematics, Uppsala University, Box. 480, SE-75106, Uppsala, SWEDEN,
email: {\tt mateusz.stroinski\symbol{64}math.uu.se}

\end{document}